\documentclass[a4paper]{article}
\usepackage[dvipsnames]{xcolor}
\usepackage[inline]{enumitem}
\usepackage{mathtools}
\usepackage[utf8]{inputenc}
\usepackage[T1]{fontenc}
\usepackage[noadjust]{cite}
\usepackage{amsmath,amsthm,amssymb,amsfonts,caption,subcaption}
\usepackage{xspace}
\usepackage{hyperref}
\usepackage[nameinlink]{cleveref}
\usepackage{verbatim}
\usepackage{tikz}
\usepackage{setspace}
\usetikzlibrary{snakes}
\usepackage[procnumbered,linesnumbered,ruled,vlined]{algorithm2e}
\usetikzlibrary{decorations.pathmorphing}
\usetikzlibrary{shapes}
\hypersetup{
	colorlinks=true,breaklinks,
	linktoc=section,
	linkcolor=Maroon,
	linkbordercolor=white,
	citecolor=dartmouthgreen,
	urlcolor=dartmouthgreen,
	pdfborder = {0 0 1}
}

\usetikzlibrary{positioning, fadings, backgrounds}                

\usepackage[textsize=footnotesize,color=green!40]{todonotes}
\usepackage[hmargin=2.5cm,vmargin=3cm]{geometry}

\usetikzlibrary{decorations.pathmorphing}
\tikzset{snake it/.style={decorate, decoration=snake}}

\newenvironment{subproof}[1][\proofname]{%
  \begin{proof}[#1]%
}{%
  \end{proof}%
}

\newcommand{\IPC}{\textsc{Isometric Path Cover}\xspace}
\newcommand{\dist}[2]{\mathsf{d}\left(#1,#2\right)}
\newcommand{\distG}[2]{\mathsf{d}_G\left(#1,#2\right)}

\newcommand{\anticp}[2]{A_{#1}\left(#2\right)}

\newcommand{\ipcor}[2]{ipco\left(#2, #1\right)}

\newcommand{\ipac}[1]{ipacc\left(#1\right)}
\newcommand{\ipco}[1]{ipco\left(#1\right)}
\newcommand{\sipco}[1]{sipco\left(#1\right)}

\newcommand{\set}[1]{\left\{#1\right\}}

\newcommand{\ie}{\textit{i.e.\xspace}}

\newcommand{\subpath}[3]{#1\left( #2,#3 \right)}

\definecolor{dartmouthgreen}{rgb}{0.05, 0.5, 0.06}

\newtheorem{theorem}{Theorem}
\newtheorem{Question}[theorem]{Question}
\newtheorem{proposition}[theorem]{Proposition}
\newtheorem{lemma}[theorem]{Lemma}
\newtheorem{observation}[theorem]{Observation}
\newtheorem{corollary}[theorem]{Corollary}

\newtheorem{definition}[theorem]{Definition}

\newtheorem{conjecture}[theorem]{Conjecture}
\newtheorem{claim}{Claim}

\newtheorem{remark}{Remark}

    \title{Strong isometric path complexity of graphs: Asymptotic minors, restricted holes, and graph operations \thanks{This research was partially financed by the IDEX-ISITE initiative CAP 20-25 (ANR-16-IDEX-0001), the International Research Center ``Innovation Transportation and Production Systems'' of the I-SITE CAP 20-25, and the ANR project GRALMECO (ANR-21-CE48-0004).}}

 \author{Dibyayan Chakraborty\footnote{School of Computer Science, University of Leeds, United Kingdom}
 \and Florent Foucaud\footnote{Université Clermont Auvergne, CNRS, Clermont Auvergne INP, Mines Saint-Etienne, LIMOS, 63000 Clermont-Ferrand, France.}}

\newcommand{\linegraph}[1]{L\left(#1\right)}
\newcommand{\powergraph}[2]{{#1}^{#2}}

\date{}
\begin{document}

\maketitle

\begin{abstract}

The (strong) isometric path complexity is a recently introduced graph invariant that captures how arbitrary isometric paths (i.e., shortest paths) of a graph can be viewed as a union of a few ``rooted" isometric paths (i.e., isometric paths with a common end-vertex). 
It is known that this parameter can be computed optimally in polynomial time.
Seemingly unrelated graph classes studied in metric graph theory (e.g.\ \emph{hyperbolic} graphs), geometric intersection graph theory (e.g.\ \emph{outerstring} graphs), and structural graph theory (e.g.\ \emph{(theta, prism, pyramid)-free} graphs) have been shown to have bounded strong isometric path complexity [Chakraborty et al., \textsc{Disc. Math. '25}], which renders this parameter promising in the realm of metric graph theory. 

We show that important graph classes studied in \emph{coarse graph theory} have bounded strong isometric path complexity. Specifically, we show that the strong isometric path complexity of $K_{2,t}$-asymptotic minor-free graphs is bounded. In fact, we prove a more general result. Let $U_t$ denote the graph obtained by adding a universal vertex to a path of $t-1$ edges. We show that the strong isometric path complexity of $U_t$-asymptotic minor-free graphs is bounded. This implies that $K_4^-$-asymptotic minor-free graphs, i.e., graphs that are quasi-isometric to a cactus [Fujiwara \& Papasoglu '23], have bounded strong isometric path complexity. On the other hand, $K_4$-minor-free graphs have unbounded strong isometric path complexity. Hence, for a graph $H$ on at most four vertices, $H$-asymptotic minor-free graphs have bounded (strong) isometric path complexity if and only if $H\not=K_4$.

We show that graphs whose all induced cycles of length at least 4 have the same length (also known as monoholed
graphs as defined by [Cook et al., \textsc{JCTB '24}]) form a subclass of $U_4$-asymptotic minor-free graphs. Hence, the strong isometric path complexity of monoholed graphs is bounded. On the other hand, we show that even-hole free graphs have unbounded strong isometric path complexity.

 We investigate which graph operations preserve strong isometric path complexity. We show that the strong isometric path complexity is preserved under the \emph{fixed power} and \emph{line graph} operators, two important graph operations. 
 We also show that the \emph{clique-sums} of finitely many graphs with strong isometric path complexity at most $k$, yield a graph with strong isometric path complexity at most $3k+18$.
\end{abstract}

 \newpage

 \section{Introduction}

All graphs considered in this paper are simple, undirected, unweighted, and finite. The study of the distance-metric of graphs is at the heart of \emph{metric graph theory}~\cite{bandelt2008metric} and finds numerous practical applications in computer science~\cite{bandelt2008metric}. To capture how arbitrary isometric paths (i.e., shortest paths) of an unweighted graph can be viewed as a union of {few} ``rooted'' isometric paths (i.e., isometric paths with a common end-vertex), a graph invariant known as \emph{isometric path complexity} was introduced in~\cite{chakraborty2026isometric},
where it was shown to have interesting properties and promising applicability. 
 Formally, it is defined as follows.

 Given a graph $G$ and a vertex $v\in V(G)$, an isometric path $P$ of $G$ is \emph{$v$-rooted} if $v$ is one of the end-vertices of $P$. The \emph{isometric path complexity} of a connected graph $G$, denoted by $\ipco{G}$, is the minimum integer $k$ such that there exists a vertex $v\in V(G)$ satisfying the following property: the vertices of any isometric path of $G$ can be covered by $k$ many $v$-rooted isometric paths. 

A stronger variant of the isometric path complexity, which is the main topic of this paper, was defined in~\cite{chakraborty2026isometric} as follows. 

\begin{definition}
    The \emph{strong isometric path complexity} of a connected graph $G$, denoted by $\sipco{G}$, is the minimum integer $k$ such that \textbf{every} vertex $v\in V(G)$ satisfies the following property: the vertices of any isometric path of $G$ can be covered by $k$ many $v$-rooted isometric paths.
\end{definition}

 We say that the isometric path complexity (resp. strong isometric path complexity) of a graph class $\mathcal{G}$ is bounded if there is a constant $k$ that depends only on $\mathcal{G}$ such that for all $G\in \mathcal{G}$, the isometric path complexity (resp. strong isometric path complexity) is at most $k$. 
Clearly, for any graph $G$, the isometric path complexity is upper bounded by the strong isometric path complexity. The gap between the isometric path complexity and the strong isometric path complexity can be arbitrarily large~\cite{chakraborty2026isometric}.

\medskip\noindent\textbf{Goal of the paper.} The (strong) isometric path complexity is a promising metric-based graph parameter, as it captures several interesting and well-studied graph classes (that seemingly did not have much in common) under the same umbrella. Moreover, they can be computed in polynomial time and have interesting algorithmic applications. (See the next paragraph for details.) Moreover, these concepts find relevance in the recently introduced \emph{coarse graph theory}. Hence, the properties of these parameters deserve to be studied further. In particular, the frontier between graph classes with bounded and unbounded values of the parameter, as well as the behavior of standard graph operations with respect to the parameters, are of interest. Our goal is thus to conduct such studies, with a particular focus on metric-based graph classes and graph operations, in order to extend the applicability of these concepts.

\medskip\noindent\textbf{Related work.} Originally, the notion of isometric path complexity was conceived to propose \emph{approximation algorithms} for the problem \textsc{Isometric Path Cover}~\cite{c22isom}, (also known as \textsc{Shortest Path Cover}~\cite{FERNAU2025115029}) where the objective is to find a minimum-cardinality set of isometric paths so that each vertex of the graph belongs to at least one of the solution paths. \textsc{Isometric Path Cover} has been studied recently in the algorithmic graph theory community~\cite{c22isom,chakraborty2026isometric,MFCS2024} and finds applications in machine learning~\cite{TG21}, \emph{cop decompositions} (in relation with the celebrated Cops and Robber game~\cite{cop-decs}), bus route design~\cite{c22isom} etc. 
It was shown that \textsc{Isometric Path Cover} admits a polynomial-time $c$-approximation algorithm on graphs with isometric path complexity at most $c$~\cite{chakraborty2026isometric} (while being NP-hard even on chordal graphs, which have isometric path complexity at most~4). Unlike some other metric graph parameters (e.g.\ \emph{treelength}), the {(strong) isometric path complexity can be computed optimally in polynomial time~\cite{chakraborty2026isometric}.

From the structural perspective, 
it turns out that many important graph classes such as \emph{hyperbolic} graphs, \emph{outerstring} graphs, and \emph{(theta, prism, pyramid)}-free graphs have bounded strong isometric path complexity~\cite{chakraborty2026isometric}. 
These graph classes are important on their own and are incomparable to each other with respect to inclusion. To our knowledge, having bounded isometric path complexity is the first known interesting common feature shared by these seemingly unrelated graph classes.

The notion of \emph{$\delta$-hyperbolicity}~\cite{gromov1987} measures how ``treelike'' a graph is from a metric point of view (we omit the formal definition in this paper). Graphs with constant $\delta$-hyperbolicity are called \emph{hyperbolic} graphs. From a theoretical perspective, many popular graph classes like interval graphs, chordal graphs, $\alpha_i$-metric graphs \cite{dragan2024alpha-i}, graphs with bounded tree-length~\cite{dourisboure2007tree}, link graphs of simple polygons ~\cite{chepoi2008diameters} have constant $\delta$-hyperbolicity. From a practical perspective, the study of $\delta$-hyperbolicity of graphs is motivated by the fact that many real-world graphs are tree-like \cite{AbuDra16,AdcSulMah13,JonLahBon08} or have small $\delta$-hyperbolicity \cite{BorCouCre+15,EdwardsKennedySanie2018,NaSa}. 
However, hyperbolic graphs exclude arbitrarily large cycles as \emph{isometric subgraphs}.

Outerstring graphs (and its subclasses) are studied in geometric intersection graph theory and computational geometry~\cite{an2024sparse,biedl2018size, bose2022computing,rok2019outerstring}. Popular graph classes like outerplanar graphs and circle graphs are subclasses of outerstring graphs. The class of (theta, prism, pyramid)-free graphs has been studied in structural graph theory~\cite{abrishami2022graphs,chudnovsky2024tree} as a way to extend studies of hereditary graph classes, based on the vast literature around perfect graphs, cutsets, decomposition theorems, etc.

Note that (as opposed to hyperbolic graphs) outerstring graphs and (theta, prism, pyramid)-free graphs may contain arbitrarily many large isometric cycles, although in a structured manner. See \Cref{fig:subclass} for graph classes with (un)bounded strong isometric path complexity. 
In this paper, we show that the strong isometric path complexity of graph classes important in \emph{coarse graph theory} is bounded.

\begin{figure}[t]
\centering
\scalebox{0.55}{\begin{tikzpicture}[node distance=7mm]

\tikzstyle{mybox}=[fill=white,line width=0.5mm,rectangle, minimum height=.8cm,fill=white!70,rounded corners=1mm,draw];
\tikzstyle{myboxg}=[fill=gray,line width=0.5mm,rectangle, minimum height=.8cm,fill=gray!40,rounded corners=1mm,draw]

\tikzstyle{myedge}=[line width=0.5mm]

\tikzstyle{myboxel}=[fill=white,line width=0.5mm,ellipse, minimum height=.8cm,fill=white!70,rounded corners=1mm,draw];

\newcommand{\tworows}[2]{\begin{tabular}{c}{#1}\\{#2}\end{tabular}}


    \node[mybox, line width=1mm] (ipacc) {\begin{tabular}{c}
          Bounded strong isometric path 
          complexity
     \end{tabular} };

     \node[mybox] (induced-minor) [below =of ipacc,yshift=-0.5cm] {\begin{tabular}{c}
          $U_t$-asymptotic minor-free \textbf{*}
     \end{tabular}}  edge[myedge] (ipacc);
     
     \node[mybox] (hyp) [below left=of ipacc,xshift=-1cm, yshift=-0.75cm] {Bounded hyperbolicity~\cite{chakraborty2026isometric} } edge[myedge] (ipacc);

     \node[myboxg] (k4m) [left=of ipacc,xshift=-2cm] {\begin{tabular}{c}
          $K_4$- minor-free~\cite{chakraborty2026isometric}
     \end{tabular}};

     \node[myboxg] (perfect) [left=of k4m,xshift=-1cm] {\begin{tabular}{c}
          Perfect
     \end{tabular}} ;

\node[myboxg] (bipartite) [below=of perfect,yshift=-0.55cm] {\begin{tabular}{c}
          Bipartite
     \end{tabular}} edge[myedge] (perfect);
     
     \node[myboxg] (even) [right=of ipacc,xshift=1.5cm] {\begin{tabular}{c}
          Even-hole free *
     \end{tabular}};
     
     \node[mybox] (theta) [below =of induced-minor] {\begin{tabular}{c}
          ($\ell$-theta, $\ell$-prism, $\ell$-pyramid)-free~\cite{chakraborty2026isometric}
     \end{tabular}}  edge[myedge] (induced-minor);
     
    \node[mybox] (outer) [below =of theta, yshift=-0.5cm] {\begin{tabular}{c}
          Outerstring~\cite{chakraborty2026isometric}
          
     \end{tabular}} edge[myedge] (theta) ;
     
    \node[mybox] (outerpl) [below =of outer, xshift=-4cm, yshift=0.35cm] {\begin{tabular}{c}
          Outerplanar~\cite{chakraborty2026isometric}
          
     \end{tabular}} edge[myedge] (outer) ;
     
     \node[mybox] (treelength) [below=of hyp,] {Bounded tree-length~\cite{c22isom}} edge[myedge] (hyp);
     
     \node[mybox] (chordality) [below=of treelength, yshift=-0.4cm] {Bounded chordality~\cite{c22isom}} edge[myedge] (treelength);
  
     \node[mybox] (bdiameter) [below=of treelength, xshift=-4.5cm, yshift=-0.4cm] {Bounded diameter~\cite{c22isom}} edge[myedge] (treelength);

    \node[mybox] (same-hole) [below right=of induced-minor,xshift=1.5cm, yshift=-0.25cm] {\begin{tabular}{c}
          Monoholed \textbf{*}
     \end{tabular}}  edge[myedge] (induced-minor);

    \node[mybox] (chordal) [below=of outer, yshift=-0.5cm] {Chordal~\cite{c22isom}} edge[myedge] (outer); \draw[myedge] (chordal.west) -| (chordality.south) ; \draw[myedge] (chordal.east) -| (same-hole.south) ;

     \node[myboxel] (line) [above left=of ipacc,xshift=-1cm, yshift=0.5cm] {\begin{tabular}{c}
          Line graph \textbf{*}
     \end{tabular}}  edge[snake, line width=0.5mm] (ipacc);

     \node[myboxel] (pow) [above =of ipacc, yshift=0.5cm] {\begin{tabular}{c}
          Graph powers \textbf{*}
     \end{tabular}}  edge[snake,line width=0.5mm] (ipacc);


\node[myboxel] (cliquesum) [above right=of ipacc, xshift=1cm, yshift=0.5cm] {\begin{tabular}{c}
         Clique-sum \textbf{*}
     \end{tabular}}  edge[snake,line width=0.5mm] (ipacc);

     \draw[myedge]     (outerpl.north) 
     |- (k4m.east);
     \draw[myedge]     (chordal.east) -| (even.south);
     \draw[myedge]     (chordal.west) --++ (-13,0) |- (perfect.west);
  \end{tikzpicture}}
  \caption{Inclusion diagram for graph classes.
If a class $A$ has an upward path to class $B$, then $A$ is included in $B$. Constant bounds for the strong isometric path complexity
on graph classes and graph operations marked with \textbf{*} are contributions of this paper. The boxes marked in gray indicate graph classes with unbounded isometric path complexity. Elliptical boxes represent graph operations.}\label{fig:subclass}
  \end{figure}

\medskip\noindent\textbf{Asymptotic minors and quasi-isometry.} Recently, Georgakopoulos and Papasoglu~\cite{georgakopoulos2023graph} stated a number of results and conjectures about \emph{coarse} properties of graphs, establishing a basis of a new area called \emph{coarse graph theory}. In this area, the objective is to ``look at graphs from far away'' and to understand their ``large-scale geometry''. Many problems and questions posed in~\cite{georgakopoulos2023graph} involve \emph{asymptotic minors} and \emph{fat minors}, and have garnered recent attention~\cite{fujiwara2023coarse, albrechtsen2023structural, albrechtsen2024characterisation, albrechtsen2024asymptotic,davies2024fat}. 
A graph $G$ contains a graph $H$ as a \emph{$K$-fat minor} if it is possible to find in $G$, (i) a collection of connected subgraphs called \emph{branch sets} corresponding to $V(H)$, and a collection of paths called \emph{branch paths} corresponding to $E(H)$ such that (i) after contracting each branch set to a vertex, and each branch path to an edge, we obtain a copy of $H$, and (ii) the above sets are at pairwise distance at least $K$, except for incident branch set–branch paths pairs, which must be at distance $0$ by definition. A graph class $\cal G$ contains $H$ as an \emph{asymptotic minor} if for every positive integer $K$, there exists a graph $G\in \cal G$, such that $G$ contains $H$ as a $K$-fat minor. Otherwise, $\cal G$ is \emph{$H$-asymptotic minor-free}.
The structure of $H$-asymptotic minor-free graphs has been studied in particular when $H=K_4$~\cite{albrechtsen2024asymptotic}, $H=K_4^-$~\cite{fujiwara2023coarse,albrechtsen2024asymptotic}, and $H=K_{2,3}$~\cite{chepoi2012constant}.

Quasi-isometry\footnote{ An \emph{$(M,A)$-quasi-isometry} between graphs $G$ and $H$ is a map $f\colon V(G)\rightarrow V(H)$ such that the following holds for fixed constants $M\geq 1, A\geq 0$, \begin{enumerate*}[label=(\alph*)]
		\item $M^{-1}\dist{x}{y}-A \leq \dist{f(x)}{f(y)} \leq M\dist{x}{y}+A$ for every $x,y\in V(G)$; and \item for every $z\in V(H)$ there is $x\in V(G)$ such that $\dist{z}{f(x)}\leq A$. 
	\end{enumerate*}
	The graphs $G$ and $H$ are \emph{quasi-isometric} if there exists a quasi-isometric map $f$ between $V(G)$ and $V(H)$. } between graph classes is an important notion in coarse graph theory. Quasi-isometry is a generalisation of {biLipschitz} maps that allows additive distortion. Embeddings with multiplicative (i.e. {biLipschitz} embeddings) or additive distortion are an area of active research~\cite{georgakopoulos2023graph,fujiwara2023coarse, albrechtsen2023structural, albrechtsen2024characterisation,fox2019embedding, dragan2006distance, chepoi2000note}. Hyperbolic graphs are closed under quasi-isometry~\cite{gromov1987}. For a fixed graph $H$, the property of excluding $H$ as an asymptotic minor is invariant under applying quasi-isometries~\cite{georgakopoulos2023graph}. On the contrary, the property of bounded (strong) isometric path complexity is not preserved under quasi-isometry (see \Cref{rem:quasi-isometry}). However, this notion has been used in some proofs regarding quasi-isometry between graph classes. For example, the notion of strong isometric path complexity has been used to show that given a $K_{2,3}$-induced minor-free\footnote{A graph $G$ contains a graph $H$ as an induced minor, if $H$ can be obtained from $G$ by deleting vertices and contracting edges. Otherwise, $G$ is $H$-induced minor-free.} graph $G$, it is possible to construct a graph $H$ in $O(nm)$-time such that $G$ is quasi-isometric to $H$ with a constant additive distortion and $H$ has treewidth at most two~\cite{chakraborty2025k}. In the special case of \emph{universally signable} graphs~\cite{vuskovic2013world}, the above quasi-isometry can be obtained in linear time. This further implies a truly-subquadratic additive approximation algorithm for \textsc{Diameter}\footnote{\textsc{Diameter} is the algorithmic problem of finding the longest isometric path in a graph.} on universally signable graphs. In contrast, assuming the \emph{Strong Exponential Time Hypothesis} (\textsc{SETH}), the diameter of split graphs (a very restricted class of universally signable graphs), cannot be computed optimally in truly sub-quadratic time \cite{borassi2016into}. However, the fact that $K_{2,3}$-induced minor-free graphs (more generally $K_{2,3}$-asymptotic minor-free graphs) are quasi-isometric to a cactus with additive distortion can be alternatively derived using known results~\cite{fujiwara2023coarse,chepoi2012constant}.

In this paper, we consider a superclass of $K_{2,3}$-asymptotic minor-free graphs and $K_4^-$-asymptotic  minor-free graphs and show that its strong isometric path complexity is bounded. Let $U_t$ denote the graph obtained by taking a path with $t$ vertices and introducing a universal vertex. See \Cref{fig:U_t}(a) for an example. We prove the following.

\begin{figure}
\centering
    \begin{tabular}{cc}
     \begin{tikzpicture}
			\foreach \x/\y [count = \n] in {0/0,1/0,2/0,3/0,4/0,2/1}
			{
				 \filldraw (\x, \y) circle (2pt);
				 
			}
			
			\foreach \x/\y/\w/\z [count = \n] in {0/0/1/0,1/0/2/0,2/0/3/0,3/0/4/0}
			{
				\draw (\x,\y) -- (\w,\z);
				\draw (2,1) -- (\x,\y);
			}
			
			\draw (2,1) -- (4,0);
	\end{tikzpicture}    & 
         \begin{tikzpicture}
            \foreach \x/\y [count = \n] in {0/0, 3/0, 1/1, 2/1, 1/0, 2/0, 1/-1, 2/-1}
			{
				 \filldraw (\x, \y) circle (2pt);
			}
            
            \foreach \x/\y/\w/\z [count = \n] in {0/0/1/1, 1/1/2/1, 0/0/2/0, 0/0/1/-1, 1/-1/2/-1}
            {
                \draw[thick] (\x,\y) -- (\w,\z);
            }
            \foreach \x/\y/\w/\z [count = \n] in {2/1/3/0, 2/0/3/0, 2/-1/3/0}
            {
                \draw[dashed] (\x,\y) -- (\w,\z);
            }
            \node[below] at (1.5,0) {$\vdots$};
        \end{tikzpicture} \\
        (a) & (b)
    \end{tabular}
	\caption{(a) The graph $U_5$. (b) The graph $F_5$. The dashed edges indicate paths.}\label{fig:U_t}
\end{figure}

\begin{theorem}\label{thm:induced}
There is a function $f\colon \mathbb{N}\times \mathbb{N}\rightarrow \mathbb{N}$ such that the strong isometric path complexity of graphs that do not contain $U_t$ as a $K$-fat minor is at most $f(K,t)$.
\end{theorem}

Since (theta, prism, pyramid)-free graphs and outerstring graphs are $U_3$-asymptotic minor-free, the above theorem generalises two results of \cite{chakraborty2026isometric}. 
It is known that $K_3$-asymptotic minor-free graphs have bounded treelength~\cite{georgakopoulos2023graph,dourisboure2007tree}, and therefore have bounded strong isometric path complexity. The above theorem implies that $U_3$-asymptotic minor-free graphs have bounded strong isometric path complexity. On the other hand, it is known from \cite{chakraborty2026isometric} (and also implied from \Cref{thm:even}) that $K_4$- minor-free graphs have unbounded (strong) isometric path complexity. Hence, we have the following corollary.

\begin{corollary}
Let $H$ be a graph with at most four vertices. Then $H$-asymptotic minor-free graphs have bounded strong isometric path complexity if and only if $H\not=K_4$.
\end{corollary}

Using a result from~\cite{chakraborty2026isometric}, \Cref{thm:induced} also implies the following corollary.

\begin{corollary}
\textsc{Isometric Path Cover} admits a constant-factor approximation algorithm on $U_t$-asymptotic minor-free graphs for all $t\geq 1$.
\end{corollary}

We note that $K_{2,t}$-asymptotic minor-free graphs are also $U_t$-asymptotic minor-free graphs. See \Cref{prp:K2-t} for a proof. Hence, \Cref{thm:induced} implies the following.

\begin{corollary}
	There is a function $f\colon \mathbb{N}\times \mathbb{N}\rightarrow \mathbb{N}$ such that the strong isometric path complexity of graphs that do not contain $K_{2,t}$ as a $K$-fat minor is at most $f(K,t)$.
\end{corollary}

Note that the class of $K_{2,3}$-asymptotic minor-free graphs and its subclasses (e.g.\ \emph{$K_{2,3}$-induced minor-free} graphs, \emph{universally-signable} graphs, \emph{outerstring} graphs) have been studied in the literature~\cite{chepoi2012constant,bose2022computing,an2024sparse,conforti1997, milanivc2024bisimplicial,dallard2024detecting}. 
Furthermore, $K_{2,t}$-asymptotic minor-free graphs are strictly included in the class of $U_t$-asymptotic minor-free graphs. For an integer $k\geq 2$, let $F_k$ denote the graph obtained by taking $k$ copies of paths of length $k$ and identifying the end-vertices of the paths. See \Cref{fig:U_t}(b). Let $\mathcal{ F}=\displaystyle\bigcup\limits_{k\geq 2 } F_k$. Clearly, $\cal F$ is $U_t$-asymptotic minor-free but contains $K_{2,t}$ as an asymptotic minor for all $t\geq 1$. Observe that all graphs in $\cal F$ are \emph{monoholed} i.e. the lengths of all induced cycles in the graph are the same. 

\medskip\noindent\textbf{Monoholed graphs and even-hole free graphs.} A \emph{hole} is an induced cycle of length at least~4. The study of graphs whose hole lengths are restricted has been pursued in Structural Graph Theory~\cite{cook2024graphs,horsfield2022structural,woodroofe2009vertex}. A graph $G$ is \emph{$\ell$-holed} if all holes of $G$ are of length $\ell$. A graph $G$ is \emph{monoholed} if it is $\ell$-holed for some integer $\ell$. For any odd integer $\ell$, $\ell$-holed graphs are also \emph{even-hole} free graphs (i.e. graphs with no holes of even length) and, for even integers $\ell$, it is known that $\ell$-holed graphs are \emph{perfect}~\cite{horsfield2022structural}, that is, do not contain any odd holes or complements of an odd hole as induced subgraphs. Studies on even-hole free graphs~\cite{vuvskovic2010even} and perfect graphs~\cite{chudnovsky2006strong} have yielded many powerful decomposition theorems and algorithmic techniques {for fundamental problems such as \textsc{Coloring}, \textsc{Independent Set}, and \textsc{Clique}.}

A complete structural description of $\ell$-holed graphs with $\ell\geq 7$ {was obtained} by Cook et al.~\cite{cook2024graphs}. Monoholed graphs can be recognised in polynomial time~\cite{cook2024graphs,horsfield2022structural}. The class of $5$-holed graphs was studied in the context of algebraic combinatorics and commutative algebra~\cite{woodroofe2009vertex}. However, a complete description of $5$-holed graphs remains unknown~\cite{cook2024graphs}. Several subclasses of $5$-holed graphs have been studied in the literature~\cite{foley2020intersection,penev2020clique}. We show (without using the decomposition theorem from~\cite{cook2024graphs}) that monoholed graphs have bounded strong isometric path complexity.

\begin{theorem}\label{thm:mono}
 The class of monoholed graphs is $U_4$-asymptotic minor-free. Hence, {monoholed graphs have bounded strong isometric path complexity.}

\end{theorem}
  
The above theorem implies that \IPC admits a constant-factor approximation algorithm on monoholed graphs. Some other algorithmic results on monoholed graphs are known. \textsc{Maximum Independent Set} can be solved in polynomial time on monoholed graphs~\cite{cook2024graphs}. \textsc{Clique} can be solved in polynomial time for $\ell$-holed graphs with $\ell\geq 5$, as there are only polynomially many maximal cliques in such graphs. The computational complexity of \textsc{Coloring} is open on these graphs. Already, optimally coloring the special case of rings (a special type of monoholed graphs that is important in their study) is a difficult (but solved) problem~\cite{maffray2021coloring}.

Next, we consider two superclasses of monoholed graphs: perfect graphs and even-hole free graphs. Bipartite graphs {(which form a subset of both these classes)} have unbounded isometric path complexity. Indeed, it can be easily checked that the isometric path complexity of an $(n\times n)$-grid is $\Omega(n)$. (Indeed, for each vertex $v$ of a $(n\times n)$-grid there exists an isometric path $P_v$ containing $\Omega(n)$ many vertices which are equidistant from $v$ and no two such vertices can be covered by one $v$-rooted isometric path.) Hence, perfect graphs have unbounded isometric path complexity. We show that the isometric path complexity of even-hole free graphs is also unbounded.

\begin{theorem}\label{thm:even}
{The class of even-hole free graphs with {maximum} degree at most~$3$ has unbounded isometric path complexity.}
\end{theorem}

\medskip\noindent\textbf{Graph operations.} In this paper, we study the behaviour of (strong) isometric path complexity under popular graph operations. We prove that two popular quasi-isometric graph operations: \emph{graph powers} and \emph{line graphs} preserve the property of having bounded strong isometric path complexity. For an integer $r\geq 1$ and a graph $G$, the $r^{th}$ power of $G$, denoted as $G^r$, is the graph with vertex set $V(G)$ where two vertices are adjacent in $G^r$ if they are at distance at most $r$ in $G$. The \emph{line graph} of $G$, denoted as $\linegraph{G}$, has vertex set $E(G)$ and two vertices are adjacent in $\linegraph{G}$ if the corresponding edges share an end-vertex in $G$. Clearly, for a fixed integer $r$, any graph $G$ and its $r^{th}$ power $\powergraph{G}{r}$ are quasi-isometric. Similarly, any graph $G$ and its line graph $\linegraph{G}$ are also quasi-isometric. Line graphs and fixed powers of graphs are important in various areas of graph theory including metric graph theory~\cite{bandelt2008metric}.
We prove the following.

\begin{theorem}\label{thm:line-power}
    Let $G$ be a graph with isometric path complexity (resp. strong isometric path complexity) at most $k$. Then the following holds:
    \begin{enumerate}[label=(\alph*)]
        \item\label{it:pow} for every integer $r\geq 1$, the isometric path complexity (resp. strong isometric path complexity) {of $G^r$} is at most $k(4r^2-2r)$;
         
        \item\label{it:line} the isometric path complexity (resp. strong isometric path complexity) of $\linegraph{G}$ is at most $3k+1$. Moreover, for every $t\geq 1$ there exists a graph $H_t$ with $\ipco{H_t} = t$ and $\ipco{\linegraph{H_t}}= 2t-1$. 
    \end{enumerate}
\end{theorem}

We do not known if the dependency on $r$ is needed in Theorem~\ref{thm:line-power}(a). Note that $r$-powers of chordal graphs (for even $r$) are also chordal, and chordal graphs have strong isometric path complexity at most~4~\cite{c22isom}, so in this case, no dependency on $r$ is needed.

If two graphs $G$ and $H$ each contain cliques $C_G$ and $C_H$, respectively, of equal size, the \emph{clique-sum} of $G$ and $H$ is obtained from the disjoint union of $G$ and $H$ by identifying pairs of vertices of $C_G$ and $C_H$ to form a single shared clique. Note that for different $C_G$ and $C_H$, the resulting graph can be different. Let $G\oplus H$ denote a graph that can be obtained by clique sum of $G$ and $H$. Clique-sums are important in diverse areas of graph theory, e.g., Wagner's Theorem on characterization of $K_5$-minor-free graphs~\cite{wagner37}, graph minor theory~\cite{diestel2025graph} etc. This operation is also studied in the context of \emph{decomposition theorems} (e.g.\ for chordal graphs, universally-signable graphs~\cite{conforti1997} etc.). The clique-sum of two hyperbolic graphs (resp. $H$-asymptotic minor-free graphs) yields a hyperbolic graph (resp. $H$-asymptotic minor-free graphs). 
We prove a similar result for graphs with bounded strong isometric path complexity. 
For graphs $G_1,G_2,\ldots,G_t$, denote by $\displaystyle\bigoplus\limits_{i=1}^t G_i$ the graph obtained by taking pairwise clique sums of $G_1,G_2,\ldots,G_t$. (Note that it is not necessary that all graphs are identified using a single clique.) 

\begin{theorem}\label{thm:clique-sum}
    Let $G_1,G_2,\ldots,G_t$ be graphs with $\sipco{G_i}\leq k$ for all $i\in [t]$. Then the strong isometric path complexity of $G=\displaystyle\bigoplus\limits_{i=1}^t G_i$ is at most $3k+18$.
\end{theorem}

We suspect that the factor of~$3$ is not needed in the bound of Theorem~\ref{thm:clique-sum}. Can it be improved to $k+c$, for some constant $c$? Note that there are graphs whose clique-sum has strictly larger strong isometric path complexity than the original graphs. For example, the strong isometric path complexity of chordal graphs (which can be constructed using clique-sums of complete graphs) is at most~4, whereas complete graphs have strong isometric path complexity~2.

\medskip\noindent\textbf{Organization of the paper.} In Section~\ref{sec:prelim}, we recall some definitions and some results. In \Cref{sec:induced}, we prove \Cref{thm:induced,thm:mono,thm:even}. In \Cref{sec:graph-op}, we prove \Cref{thm:line-power,thm:clique-sum}. Finally, we conclude in \Cref{sec:conclude}.

\section{Preliminaries}\label{sec:prelim}
In this section, we recall some definitions and some related observations.
A sequence of distinct vertices form a \emph{path} $P$, if any two consecutive vertices are adjacent. 
Whenever we fix a path $P$ {in} $G$, we shall refer to the subgraph formed by the edges between the consecutive vertices of $P$. {The \emph{length} of a path $P$ is the number of edges in $P$}. A path is \emph{induced} if there are no graph edges joining non-consecutive vertices. A path is \emph{isometric} if it is a shortest path between its end-vertices. For two vertices $u,v$ of a graph $G$, $\dist{u}{v}$ denotes the length of an isometric path between $u$ and $v$. An isometric path $P$ is $v$-rooted if one of the end-vertices of $P$ is $v$. 

For a path $P$ of a graph $G$ between two vertices $u$ and $v$, the vertices $V(P)\setminus \{u,v\}$ are \emph{internal vertices} of $P$. A path between two vertices $u$ and $v$ is called a $(u,v)$-path. Similarly, we have the notions of \emph{isometric $(u,v)$-path} and \emph{induced $(u,v)$-path}. For a vertex $r$ of $G$ and a set $S$ of vertices of $G$, the \emph{distance of $S$ from $r$}, denoted as $\dist{r}{S}$, is the minimum of the distances between any vertex of $S$ from $r$, that is, $\dist{r}{S}=\min\{\dist{r}{v}\colon v\in S\}$. 
For a subgraph $H$ of $G$, the \emph{distance of $H$ from $r$} is $\dist{r}{V(H)}$. For two sets $S$ and $T$, $\dist{S}{T}=\min\{ \dist{r}{T}\colon r\in S\}$.

\subsection{Asymptotic minors}

For a positive integer $K$ and a graph $H$, a \emph{$K$-fat minor model} of $H$ in a graph $G$ is a collection $\mathcal{M}=\left(B_v\colon v\in V(H)\right) \cup \left(P_e\colon e\in E(H)\right)$ of connected subgraphs of $G$ such that \begin{itemize}
    \item $V(B_v) \cap V(P_e)\neq \emptyset$ whenever $v$ is an end-vertex of $e$ in $H$;
    \item for any pair of distinct $X,Y\in \mathcal{M}$ not covered by the above condition, we have $\distG{X}{Y} \geq K$. 
\end{itemize}  

The subgraphs in $(B_v\colon v\in E(H))$ are the ``branch sets'' of the model. The subgraphs in $(P_e\colon e\in E(H))$ are the ``branch paths'' of the model. If $G$ contains a $K$-fat minor model of a graph $H$, then we say that $H$ is a $K$-fat minor of $G$. Note that, for $e\in E(H)$, we can assume that the connected subgraphs $P_e$ are indeed paths.

\begin{definition}[\cite{georgakopoulos2023graph}]
    A graph class $\cal G$ contains $H$ as an \emph{asymptotic minor} if for every integer $K\geq 1$, there exists a graph $G\in \cal G$, such that $G$ contains $H$ as a $K$-fat minor.
\end{definition}

We say that {a graph class} $\mathcal{G}$ is \emph{$H$-asymptotic minor-free} if $\cal G$ does not contain $H$ as an asymptotic minor. In other words, $\cal G$ is $H$-asymptotic minor-free if there exists a constant $K$, depending only on $\mathcal{G}$, such that no graph in $\mathcal{G}$ contains a $K$-fat minor model of $H$. 

\begin{figure}
    \centering
    \begin{tikzpicture}
        
        \foreach \x/\y/\w/\z [count=\n] in {0/0/-3/-3, 0/0/-1/-3, 0/0/1/-3,  -3/-3/-1/-3, -1/-3/1/-3, 1/-3/3/-3,3/-3/0/0}
        {
            \draw (\x,\y) -- (\w,\z);
        }
        \foreach \x/\y [count=\n] in {0/0, -3/-3, -1/-3, 1/-3, 3/-3}
        {
            \filldraw[fill=white, dashed] (\x,\y) ellipse (0.5cm and 0.25cm);
        }

        \draw[dashed] (-3.75,-3.5) rectangle (3.75,-2.5);
        \draw[dashed] (-0.75,-0.5) rectangle (0.75,0.5);
        \draw[dashed] (-2.1,-1.75) rectangle (-1.3,-1.5);
        \draw[dashed] (-0.85,-1.75) rectangle (-0.2,-1.5);
        \draw[dashed] (0.4,-1.75) rectangle (0.7,-1.5);
        \draw[dashed] (1.3,-1.75) rectangle (1.9,-1.5);

        \node at (0,0) {$B_u$};
        \node[left] at (-0.75, -0.25) {$B'_0$}; 
        \node[left] at (-2.1, -1.65) {$X_v$}; 
        \node[left] at (-3.75, -3.25) {$B'_1$}; 
        
    \end{tikzpicture}
    \caption{{Illustration of the proof of \Cref{prp:K2-t}. Dashed ellipses denote the branch sets of $\mathcal{M}_t$. The solid lines indicate the branch paths of $\mathcal{M}_t$. The dashed rectangular boxes indicate the branch sets of the $\left\lfloor\frac{K}{4}\right\rfloor$-fat minor model of $K_{2,t}$.}}
    \label{fig:U_tK_2t}
\end{figure}

\begin{proposition}\label{prp:K2-t}
For an integer $t\geq 1$, if a graph class $\cal G$ is $K_{2,t}$-asymptotic minor-free, then $\cal G$ is $U_{t}$-asymptotic minor-free. 
\end{proposition}
\begin{proof}
 See \Cref{fig:U_tK_2t} for illustration. Assume for contradiction, that $\cal G$ is not $U_t$-asymptotic minor-free. This implies that for every integer $K\geq 12$, there exists a graph $G\in \cal G$ containing a $K$-fat minor model $\mathcal{M}_t$ of $U_t$. Let $\mathcal{M}_t=\left(B_w\colon w\in V(U_t)\right) \cup \left(P_e\colon e\in E(U_t)\right)$. We shall construct a $\left\lfloor\frac{K}{4}\right\rfloor$-fat minor model of $K_{2,t}$. 
    Let $u$ denote the universal vertex of $U_t$ and $B_u$ denote the corresponding branch set in $\mathcal{M}_t$. For each edge $uv\in E(U_t)$, let $Q_{uv}\subset V(P_{uv})$ denote an induced path of length at least $K$ such that $Q_{uv}$ meets $B_u$ or $B_v$ only at its end-vertices. By definition, such a path always exists. For each edge $uv\in E(U_t)$, let $X_v$ denote the subpath of $Q_{uv}$ whose length is exactly $\left\lfloor\frac{K}{4}\right\rfloor+3$ and is at distance at least $\left\lfloor\frac{K}{4}\right\rfloor$ from both of the end-vertices of $Q_{uv}$. Define $P^1_v, P^2_v$ to be the two paths obtained by deleting the vertices of $X_v$ from $Q_{uv}$. Define $B'_0=B_u$. Finally, let $B'_1$ be the graph induced by the vertices $$\displaystyle\left(\bigcup\limits_{w\in V(U_t)\setminus \{u\}} V(B_w)\right) \cup \left(\displaystyle\bigcup\limits_{e\in E(U_t-\{u\})} V(P_{e})\right)$$ It is easy to check that $\{B'_0,B'_1\} \cup \{X_v\colon uv\in E(U_t)\}\cup \{P^1_v,P^2_v\colon uv\in E(U_t)\}$ is a $\left\lfloor\frac{K}{4}\right\rfloor$-fat minor model of $K_{2,t}$.
\end{proof}

\subsection{Isometric path complexity} 
Let $G$ be a graph, and $r \in V(G)$ be a vertex. In this section, $\dist{.}{.}$ refers to the metric in $G$. Let $G_r$ be the graph obtained by removing all edges $xy$ from $G$ such that $\dist{r}{x}=\dist{r}{y}$. Then, for each edge $e=xy\in E(G_r)$ with $\dist{r}{x} = \dist{r}{y} - 1$, orient $e$ from $y$ to $x$. Let $\overrightarrow{G_r}$ be the directed acyclic graph formed after applying the above operation on $G'$. Note that this digraph can easily be computed in linear time using a Breadth-First Search (BFS) traversal with starting vertex $r$. 

\sloppy  {The following definition in \cite{chakraborty2026isometric,c22isom} is inspired by the terminology of posets (as the graph $\overrightarrow{G_r}$ can be seen as the Hasse diagram of a poset).} For a graph $G$ and a vertex $r\in V(G)$, two vertices $x,y\in V(G)$ are \emph{antichain vertices} if there are no directed paths from $x$ to $y$ or from $y$ to $x$ in $\overrightarrow{G_r}$. A set $X$ of vertices of $G$ is an \emph{antichain set} if any two vertices in $X$ are antichain vertices. 
For a subgraph $H$ of $G$, $\anticp{r}{H}$ shall denote a maximum antichain set of $H$ in $\overrightarrow{G_r}$. Let $\ipac{\overrightarrow{G_r}}=\max\set{|\anticp{r}{P}|\colon~P~\text{is an isometric path in }G}.$ Define $\ipac{G}=\min \set{\ipac{\overrightarrow{G_r}}\colon r\in V(G)}.$ 
Let $\ipcor{r}{G}$ denote the minimum integer $k$ such that any isometric path $P$ of $G$ can be covered by $k$ many $r$-rooted isometric paths in $G$. 
Recall that, $\ipco{G}=\min \{ \ipcor{r}{G} \colon r\in V(G)\}$ and $\sipco{G}=\max\{ \ipcor{r}{G}\colon r\in V(G) \}$.
Using Dilworth's Theorem, {one can} prove the following important lemma and proposition.

\begin{lemma}[\cite{chakraborty2026isometric}]\label{lem:ipac-ipco}
    For any graph $G$ and vertex $r$, $\ipcor{r}{G} = \ipac{\overrightarrow{G_r}}$. Therefore, $\ipco{G}=\ipac{G}$.
\end{lemma}

{Thanks to Lemma~\ref{lem:ipac-ipco}, we may use one or the other notion depending on which is easier to use in a given proof.} We shall use the following proposition. We provide its proof for completeness. 

\begin{proposition}[\cite{chakraborty2026isometric,c22isom}]\label{prp:antichain-length}
Let $G$ be a graph and $r$, an arbitrary vertex of $G$. Consider the directed acyclic graph $\overrightarrow{G_r}$, and let $P$ be an isometric path between two vertices $x$ and $y$ in $G$. Then $|E(P)|\geq |\dist{r}{x}-\dist{r}{y}| + |\anticp{r}{P}| - 1$. 
\end{proposition}
\begin{proof}
Orient the edges of $P$ from $y$ to $x$ in $G$. First, observe that $P$ must contain a set $E_1$ of oriented edges such that $|E_1|=|\dist{r}{y}-\dist{r}{x}|$ and for any $\overrightarrow{ab}\in E_1$, $\dist{r}{a}=\dist{r}{b}+1$. Let the vertices of a largest antichain set of $P$ in $\overrightarrow{G_r}$, \ie, $\anticp{r}{P}$, be ordered as $a_1,a_2,\ldots,a_t$ according to their occurrence while traversing $P$ from $y$ to $x$. For $i\in [2,t]$, let $P_i$ be the subpath of $P$ between $a_{i-1}$ and $a_i$. Observe that for any $i\in [2,t]$, since $a_i$ and $a_{i-1}$ are antichain vertices, there must exist an oriented edge $\overrightarrow{b_ic_i}\in E(P_i)$ such that either $\dist{r}{b_i} = \dist{r}{c_i}$ or $\dist{r}{b_i}=\dist{r}{c_i} - 1$. Let $E_2=\{b_ic_i\}_{i\in [2,t]}$. Observe that $E_1\cap E_2=\emptyset$ and therefore $|E(P)|\geq |E_1| + |E_2| = |\dist{r}{y}-\dist{r}{x}| + |\anticp{r}{P}| - 1$.
\end{proof}

\section{Asymptotic minors and {graphs with} restricted holes} \label{sec:induced}

In \Cref{sec:asymp}, we show that $U_t$-asymptotic minor-free graphs have bounded strong isometric path complexity and prove \Cref{thm:induced}. In \Cref{sec:mono}, we show that monoholed graphs are $U_4$-asymptotic minor-free and prove \Cref{thm:mono}. Finally, in \Cref{sec:even}, we show that even-hole free graphs have unbounded isometric path complexity and prove \Cref{thm:even}.

\subsection{Asymptotic minors}\label{sec:asymp}
In this section, we prove \Cref{thm:induced}. Let $\mathcal{G}$ be a graph class that does not contain $U_t$ as an asymptotic minor, for some $t\geq 1$. Then, by the definition, there exists an integer $K$ and a graph $G\in \mathcal{G}$ such that $G$ does not contain a $K$-fat minor model of $U_t$. We can assume $K\geq 5$. The following and \Cref{lem:ipac-ipco} prove \Cref{thm:induced}.

\newcommand{\mbound}{{12K}}
\newcommand{\addbound}{{5}}

\begin{theorem}\label{lem:house-ipac}
    For any two integers $ t\geq 2, K\geq 5$, if $G$ does not contain a $K$-fat minor model of $U_t$, then for any vertex $r\in V(G)$, $\ipac{\overrightarrow{G_r}} \leq (\mbound+\addbound)t-1$. 
\end{theorem}

\medskip\noindent  We devote the remainder of this section to proving \Cref{lem:house-ipac}.  For integers $ t\geq 2, K\geq 5$, assume for contradiction that there exists a graph $G$ that does not contain a $K$-fat minor model of $U_t$, and there exists a vertex $r\in V(G)$ such that $\ipac{\overrightarrow{G_r}} \geq (\mbound+\addbound)t$.

 Observe that there exists an isometric path $P$ with $|\anticp{r}{P}| \geq (\mbound+\addbound)t$. Assume that the two end-vertices of $P$ are $u$ and $v$, respectively. Notice that $r\notin \set{u,v}$ (otherwise we would have $\anticp{r}{P}=\emptyset$). 
 Let $k=|\anticp{r}{P}|$ and the elements of $\anticp{r}{P}$ are ordered as $p_1,p_2,\ldots,p_k$ as they are encountered while traversing $P$ from $u$ to $v$. Let $a_1=p_1$, and for $i\in [2,t]$ let $n(i)=(\mbound+\addbound)(i-1)$, $a_i=p_{n(i)}$. Let $A=\set{a_i\colon i\in [t]}$. Recall that this set forms an antichain in $\overrightarrow{G_r}$, and its definition implies that any two of its vertices are far apart from each other.

\begin{figure}[t]
	\centering
    \scalebox{0.7}{
    \begin{tabular}{cc}
     \begin{tikzpicture}[scale=1.5]
	    \draw[thick] (5,0) -- (0,0);
        \draw[thick, dash dot] (2.5,4) -- (2.5,0.5);\draw[thick] (2.5,0.5) -- (2.5,0); \draw[thick, densely dotted] (2.5,4) -- (4,0); \draw[thick, dash dot] (2.5,0.5) -- (1.5,0);
            \foreach \x/\y [count = \n] in
			{0/0, 0.5/0, 1/0, 1.5/0, 2/0, 2.5/0, 3/0, 3.5/0, 4/0, 4.5/0, 2.5/3.5, 2.5/3, 2.5/2.5, 2.5/2, 2.5/1.5, 2.5/1, 2.5/0.5, 2.6667/3.5556, 2.8333/3.1111, 3.0/2.6667, 3.1667/2.2222, 3.3333/1.7778, 3.5/1.3333, 3.6667/0.8889, 3.8333/0.4444, 5/0}
			{
				 \filldraw (\x, \y) circle (1.5pt);
			}
            \node[above] at (2.5, 4.1) {$r$};  \filldraw (2.5, 4) circle (2pt);
            \node[above] at (0.75,0) {$P$};
            \node[right] at (5.1,0) {$v$};
            \node[right] at (-0.4,0) {$u$};

            \node[below] at (2.5,0) {$a_1$};
            \node[below] at (1.5,0) {$b_1$};
            \node[below] at (4.25,0) {$a_2=b_2$};
            \node[right] at (2.5,0.5) {$b'_1$};
            \node[right] at (3.8333,0.4444) {$b'_2$};
            \node[right] at (2.5,2) {$Q_1$};
            \node[right] at (3.3333,1.7778) {$Q_2$};
	\end{tikzpicture} 
    \end{tabular}}
	\caption{Illustration of the notations used to denote the paths returned by Procedure~\ref{prc:pillar} for the horizontal $(u,v)$-path $P$ (indicated by the solid black line). The path $Q_1$ (indicated with dash-dots) is the vertical $(r,b'_1)$-path with $b_1$ appended to it. The path $Q_2$ (indicated with dots) is the downward $(r,a_2)$-path. }\label{fig:fat}
\end{figure}

\newcommand{\proc}[4]{\textsc{Connect}\left(#1,#2,#3,#4\right)}

\begin{algorithm}[t]
	\renewcommand{\algorithmcfname}{Procedure}
	\newcommand{\hrulealg}[0]{\vspace{1mm} \hrule \vspace{1mm}}
	\caption{A procedure $\proc{r}{P}{a}{u}$ to build pillars. \label{prc:pillar}}
	\SetKwInOut{KwIn}{Input}
	\SetKwInOut{KwOut}{Output}
	\KwIn{A root vertex $r$, an isometric path $P$ whose end-vertices are distinct from $r$, a vertex $a\in V(P)$, and an end-vertex $u$ of $P$.}
	\KwOut{A path.}
	\hrulealg

	Let $Q$ be any $(r,a)$-isometric path in $G$;
	
	Let $Q'$ be the maximal $(r,a')$-subpath of $Q$ such that no internal vertex of $Q'$ is adjacent to any vertex of $P$.
	
	\If{$a'\in V(P)$}{
		Return $Q'$.
	}
	\Else {
		Let $a''\in V(P) \cap N(a')$ be the vertex such that no internal vertex of the $(u,a'')$-subpath of $P$ is a neighbor of $a'$. 
		
		Let $Q''$ be the path obtained by appending $a''$ to $Q'$.
		
		Return $Q''$.
	}
	
\end{algorithm} 

Now, using the vertices of $A$ and Procedure~\ref{prc:pillar}, we construct special induced paths that will aid us to construct the $K$-fat minor model of $U_t$. First, we illustrate Procedure~\ref{prc:pillar} with an example.

\medskip \noindent \textbf{Illustration of Procedure~\ref{prc:pillar}.} 
For illustration purposes, consider the graph in \Cref{fig:fat}. Let $P$ be the $(u,v)$-path shown in the figure, and consider $\proc{r}{P}{a_1}{u}$. In this case, the vertex $a$ (used in the description of Proposition~\ref{prc:pillar}) is $a_1$, Procedure~\ref{prc:pillar} returns from Line 8, and therefore $Q_1=r\ldots b'_1b_1$. Now consider $\proc{r}{P}{a_1}{u}$. In this case, Procedure~\ref{prc:pillar} returns from Line 4, $b_2=a_2$, and $Q_2$ is the downward $(r,a_2)$-path.

  \medskip \noindent \textbf{Finding special induced paths.} Now, we shall describe a general procedure to find special induced paths which will allow us to create branch paths and branch sets of a $K$-fat minor model of $U_t$ (and hence arrive to a contradiction). Consider Procedure~\ref{prc:pillar}.
  Let $Q_i$ denote the output of $\proc{r}{P}{a_i}{u}$ where $a_i\in A$.  For each $i\in [t]$, observe that one end-vertex of $Q_i$ lies on $P$. Let us denote this end-vertex as $b_i$ and its neighbor in $Q_i$ as $b'_i$. 
  The construction of Procedure~\ref{prc:pillar} implies the following claim establishing some special properties of these paths.

  \begin{claim}\label{clm:induced:1} 
     For each $i\in [t]$, the following holds:
     \begin{enumerate}[label=(\alph*)]
         \item  the path $Q_i$ is an induced path;
         \item  no vertex of $V(Q_i)\setminus \set{b_i,b'_i}$ is adjacent to a vertex of $P$;
         \item\label{clm:it:c} the $(r,b'_i)$-subpath of $Q_i$ is a subpath of an $(r,a_i)$-isometric path.
     \end{enumerate}
  \end{claim}
  \begin{subproof}
      To prove (a) fix any $i\in [t]$. The path $Q_i$ is the path returned by $\proc{r}{P}{a_i}{u}$. If $\proc{r}{P}{a_i}{u}$ returned from Step-$4$ of Procedure~\ref{prc:pillar}, then $Q_i$ is an $(r,w_i)$-isometric path and therefore is an induced path. Otherwise, $\proc{r}{P}{a_i}{u}$ must have returned from Step-$8$ of Procedure~\ref{prc:pillar}. In this case, $Q'_i$ is constructed by taking a subpath of an $(r,a_i)$-isometric path and appending to it one more vertex from the path $P$. By construction (see Step-$1$ and Step-$6$), $Q'_i$ remains an induced path. $(b)$ follows easily from the construction. To observe $(c)$, we consider two cases. If Procedure~\ref{prc:pillar} returned from Line 4, then Line 1 and 2 imply $b_i=a_i$, and the path returned by $\proc{r}{P}{a_i}{u}$ is itself an $(r,a_i)$-isometric path and thus (c) follows. Otherwise, $b_i\not= a_i$ but lies on $P$, and due to Line 6 and 7, $Q_i - b_i$ is a subpath of the $(r,a_i)$-isometric path considered in Line 1. Now (c) follows.
  \end{subproof}

The properties proved in the above claim will allow us to use the paths $\{Q_i \colon i\in [t] \}$ to construct branch paths and branch sets of $U_t$, and reach a contradiction. Note that the end-vertex of $Q_i$ that lies on $V(P)$ (i.e. $b_i$) may not be the same as $a_i$ which was initially used (in $\proc{r}{P}{a_i}{u}$) to construct $Q_i$. We prove the following claim.

\begin{figure}[t]
    \centering
    \begin{subfigure}[t]{0.45\textwidth}
    \centering
        \begin{tikzpicture}
        \draw[thick, densely dotted] (0,4) -- (0,1.5); \draw[thick, dash dot] (0,0) -- (0,1.5);\draw[thick] (0,0) -- (-0.5,0) -- (-0.5,0.5) -- (-1,0.5) -- (-1,1) -- (0,1.5);
        \draw[dashed] (-1,1) -- (-2,1);\draw[dashed] (0,0) -- (1,0);
        
            \foreach \x/\y [count = \n] in	{0/0.0, 0/0.5, 0/1.0, 0/1.5, 0/2.0, 0/2.5, 0/3.0, 0/3.5, 0/4.0, -1/1, -1/0.5, -0.5/0.5, -0.5/0, -2/1, 1/0 }
			{
				 \filldraw (\x, \y) circle (1.5pt);
			}
            \node [above] at (0,4) {$r$};
            \node [below] at (0,0) {$a_i$};
            \node [right] at (0,1.5) {$b'_i$};
            \node [left] at (-2,1) {$u$};
            \node [right] at (1,0) {$v$};
            \node [below] at (-1.25,1) {$b_i$};
        \end{tikzpicture}
        \subcaption{Illustration of the proof of \Cref{clm:lb1}. The rectilinear curve (made of solid and dashed lines) between $u$ and $v$ represents $P$. The (vertical) dotted line represents $R$. The (vertical) dash-dotted line represents $S'$ . }\label{fig:clm2.2}
    \end{subfigure}~
    \begin{subfigure}[t]{0.45\textwidth}
    \centering
            \begin{tikzpicture}
                \draw[thick, densely dotted] (0,4) -- (0,1.5);
                \draw[thick, dash dot] (0,0) -- (0,1.5);
                \draw[thick, densely dotted] (0,4) -- (-2.5,0);
                \draw[thick] (0,0) -- (-0.5,0) -- (-0.5,0.5) -- (-1,0.5) -- (-1,1); \draw[thick, densely dotted] (-1,1)  -- (0,1.5);
                \draw[thick, gray] (-1,1) -- (-2.5,0); \draw[thick, dashed] (-2.5,0) -- (-3,0);
        
        \draw[dashed] (0,0) -- (1,0);
        
            \foreach \x/\y [count = \n] in	{0/0.0, 0/0.5, 0/1.0, 0/1.5, 0/2.0, 0/2.5, 0/3.0, 0/3.5, 0/4.0, -1/1, -1/0.5, -0.5/0.5, -0.5/0, 1/0, -2.5/0, -3/0,  -2.188/0.500, -1.875/1.000, -1.563/1.500, -1.250/2.000, -0.938/2.500, -0.625/3.000, -0.313/3.500 }
			{
				 \filldraw (\x, \y) circle (1.5pt);
			}
            \node [above] at (0,4) {$r$};
            \node [below] at (0,0) {$a_j$};
            \node [right] at (0,1.5) {$b'_j$};
            \node [left] at (-3,0) {$u$};
            \node [right] at (1,0) {$v$};
            \node [above] at (-1,1) {$b_j$};
            \node [below] at (-2,0) {$a_i=b_i$};
            \end{tikzpicture}
        \subcaption{Illustration of the proof of \Cref{clm:lb2}. The dotted line between $r,b_j$ indicates $Q_j$. The dotted line between $r,b_i$ indicates $Q_i$.  The gray line and the rectilinear black curve together indicate $S$. The gray line also indicates $R_{ij}$ in this example.}\label{fig:clm2.3}
    \end{subfigure}

    \begin{subfigure}[t]{0.45\textwidth}
    \centering
        \begin{tikzpicture}
        
            \draw[thick, gray] (1,1) -- (0.65,2.05);
            \draw[thick,densely dotted] (0.5,2.5) -- (0.65,2.05);
            \foreach \x/\y [count = \n] in	{0/4, -3/0, 3/0, -2/0, 1/1, 0.5/2.5, 0.65/2.05, -0.5/3, -0.75/2.5, -1.5/1}
			{
				 \filldraw (\x, \y) circle (1.5pt);
			}

            \draw[thick]   plot[smooth, tension=1] coordinates { (-3,0) (-2,0) (1,1) (3,0)};

            \draw[thick, densely dotted] (0,4) -- (0.5,2.5);
            \draw[dashed ] (0.5,2.5) -- (-0.5,3);
            \draw[thick, densely dotted] (0,4) -- (-2,0);
            \node [above] at (0,4) {$r$};
            
             \node [right] at (0.65,2.05) {$c_i$};
             \node [right] at (0.5,2.5) {$c'_i$};
            \node [left] at (-3,0) {$u$};
            \node [right] at (3,0) {$v$};
            \node [below] at (1,1) {$b_i$};
             \node [below] at (-2,0) {$a_j$};
             \node [left] at (-0.5,3) {$d_j$};
             \node [left] at (-0.75,2.5) {$d'_j$};
             \node [left] at (-1.5,1) {$f_j$};

        \end{tikzpicture}
        \subcaption{Illustration of the proof of \Cref{clm:3}. The thick gray line indicates $Q'_i$.  The dotted line between $r$ and $a_j$ indicates $Q_j$. The dotted line between $r$ and $c_i$ indicates a subpath of $Q_i$. The gray line indicates $Q'_i$. The dashed line indicates a $(c'_i,d_j)$ isometric path.}\label{fig:clm:3}
    \end{subfigure}~
    \begin{subfigure}[t]{0.45\textwidth}
    \centering
        \begin{tikzpicture}
         \draw[gray, thick] (1,1) -- (-0.5,2);
            
            \foreach \x/\y [count = \n] in	{0/4, -1/0, -2/0, 3/0, 1/1, -0.5/2, -0.75/1, -0.875/0.5, -1.35/0 }
			{
				 \filldraw (\x, \y) circle (1.5pt);
			}

            \draw[thick]   plot[smooth, tension=1] coordinates { (-1,0) (1,1) (3,0)};

            \draw[thick] (-2,0) -- (-1,0);
            \draw[thick] (-0.875,0.5) -- (-1.35,0);
           
            \draw[thick, densely dotted] (0,4) -- (-0.875,0.6); \draw[thick] (-0.875,0.6) -- (-1,0);
            \node [above] at (0,4) {$r$};
            \node [below] at (1,1) {$c$};
            \node [left] at (-0.5,2) {$d$};
            \node [left] at (-2,0) {$u$};
            \node [right] at (3,0) {$v$};
            \node [left] at (-0.875,0.6) {$b'_i$};
             \node [right] at (-0.75,1) {$z$};
            \node [below] at (-1,0) {$a_i$};
            \node [below] at (-1.35,0) {$b_i$};
        \end{tikzpicture}
        \subcaption{Illustration of the proof of \Cref{clm:induced:2}. The dotted line along with the vertex labeled $b_i$ indicates $Q_i$. The gray line indicates a $(c,d)$-isometric path. }\label{fig:clm:induced:2}
    \end{subfigure}
    \caption{ Illustrations of notations used in the proofs for the claims in \Cref{lem:house-ipac}.}
    \label{fig:placeholder}
\end{figure}

\begin{claim}\label{clm:lb1}
    For each $i\in [t]$, let $X_i$ denote the $(a_i,b_i)$-subpath of $P$. Then $|A \cap V(X_i)|\leq 3$.
\end{claim}

\begin{subproof}
See \Cref{fig:clm2.2}. Let $R$ denote the $(r,b'_i)$-subpath of $Q_i$. Due to \Cref{clm:induced:1}\ref{clm:it:c}, we know that there is an $(r,a_i)$-isometric path $S$ such that $V(R)\subseteq V(S)$. Hence, $b'_i\in V(S)$; let $S'$ denote the $(b'_i,a_i)$-subpath of $S$. Observe that appending $b_i$ to $S'$ gives a path from $b_i$ to $a_i$. Since $\dist{b'_i}{a_i} = |\dist{r}{a_i} - \dist{r}{b'_i}|$ and $|\dist{r}{b'_i} - \dist{r}{b_i}|\leq 1$, we have, $\dist{b_i}{a_i} \leq 2 + |\dist{r}{b_i} - \dist{r}{a_i}|$. If $|A \cap V(X_i)|\geq 4$, then \Cref{prp:antichain-length} implies  $\dist{b_i}{a_i} \geq 3 + |\dist{r}{b_i} - \dist{r}{a_i}|$, a contradiction. Hence, $|A \cap V(X_i)|\leq 3$.
\end{subproof}

\begin{claim}\label{clm:lb2}
   For $\{i,j\}\subseteq [t]$, let $R_{ij}$ denote the $(b_i,b_j)$-subpath of $P$. Then $|A \cap V(R_{ij})|\geq \mbound+2$.
\end{claim}

\begin{subproof}
    See \Cref{fig:clm2.3}. Without loss of generality, assume that $i<j$. Let $S$ denote the $(a_i,a_j)$-subpath of $P$. Observe that by the definition of $a_i$ and $a_j$, $|A \cap V(S)|\geq \mbound+\addbound$. Let $X_i$ and $X_j$ denote the $(a_i,b_i)$-subpath and $(a_j,b_j)$-subpath of $P$, respectively. Due to \Cref{clm:lb1}, we have that {$|A \cap V(X_i)|\leq 3$, $|A \cap V(X_j)|\leq 3$}. According to Procedure~\ref{prc:pillar}, $b_j$ always lies in $S$. Hence we have two cases depending on whether $b_i$ lies in $S$ or not. First, consider the case when $b_i\in V(S)$. In this case, $b_i=a_i$. Indeed, according to Procedure~\ref{prc:pillar}, $b_i$ lies on the $(u,a_i)$-subpath of $P$ and therefore $b_i$ lies in $S$ only if $b_i=a_i$.
    Therefore, $S$ can be constructed by concatenating $R_{ij}$ and $X_j$. Hence, we have the following.
    \begin{align*}
         |A \cap V(R_{ij})|+ |A \cap V(X_j)|  & \geq |A \cap V(S)| \\
        |A \cap V(R_{ij})|+3 &\geq  \mbound+\addbound  \\
        |A \cap V(R_{ij})| & \geq \mbound+2
    \end{align*}

    Consider the case $b_i\not\in V(S)$. Observe that $X_j$ is a subpath of $S$, and that the subpath $S'$ of $S$ obtained by deleting the vertices of $X_j$ is a subpath of $R_{ij}$. Hence, $|A \cap V(R_{ij})| \geq |A \cap V(S')| \geq |A \cap V(S)|-|A \cap V(X_j)| \geq 12K+5$. 

\end{subproof}

The above claim implies that the $(b_i,b_j)$-subpath of $P$, with $i\neq j$, necessarily contains an antichain of large size. For each $i\in [t]$, let $Q'_i$ denote the maximal subpath of $Q_i$ such that (i) $Q'_i$ has $b_i$ as an end-vertex and (ii) $\dist{Q'_i}{Q_j}\geq K$ for all $i\neq j, j\in [t]$.
Note that $Q'_i$ is well-defined, as for all $i,j\in [t]$, $\dist{b_i}{Q_j}\geq 6K$. (Otherwise, it can be shown, using similar arguments as the ones in \Cref{clm:3}, that $\dist{b_i}{b_j} \leq 12K$, contradicting \Cref{clm:lb2} and \Cref{prp:antichain-length}. We omit these details as \Cref{clm:3} already implies that $Q'_i$ is well-defined for each $i\in [t]$.)
Let $c_i$ be the end-vertex of $Q'_i$ which is distinct from $b_i$.  Now we prove a lower bound on the length of $Q'_i$, for each $i\in [t]$. We prove the following final claim before providing the explicit constructions of the branch paths and branch sets.

  \begin{claim}\label{clm:3}
      For each $i\in [t]$, the path $Q'_i$ has length at least $5K+1$. 
  \end{claim}
\begin{subproof}
    See \Cref{fig:clm:3}. Assume for contradiction that there exists an integer $i\in [t]$ such that the length of $Q'_i$ is at most $5K$. Let $c'_i$ be the vertex adjacent to $c_i$ in $Q_i$ which is not a vertex of $Q'_i$. Let $j\neq i$ be the minimum index such that $c'_i$ has 
    distance strictly less than $K$ from some vertex $d_j\in V(Q_j)$. Let $T_{ij}$ denote a $(c'_i,d_j)$-isometric path in $G$. Let $d'_j$ be the vertex of $Q_j$ such that $\dist{r}{d'_j} = \max \{\dist{r}{c'_i}, \dist{r}{d_j}\}$. In other words, $d'_j$ is the vertex of $Q_j$ which is either $d_j$ or lies at the same distance from $r$ as $c'_i$.
    Let $f_j$ be the vertex of $Q_j$ such that $\dist{r}{f_j} = \max \{\dist{r}{b_i}, \dist{r}{b_j}\}$. In other words, $f_j$ is the vertex of $Q_j$ which is either $b_j$ or lies at the same distance from $r$ as $b_i$. Observe that $$ \dist{b_i}{b_j} \leq \dist{b_i}{c'_i} + \dist{c'_i}{d_j} + \dist{d_j}{d'_j} + \dist{d'_j}{f_j} + \dist{f_j}{b_j}.$$ 
    Recall 
    that $\dist{b_i}{c'_i}\leq 5K+1$ and $\dist{c'_i}{d_j}\leq K-1$. By the definition of $d'_j$ and $f_j$, these imply that $\dist{d'_j}{f_j} \leq 5K+1$ and $\dist{d_j}{d'_j}\leq K-1$. Finally, the distance between $f_j$
    and $b_j$ is $|\dist{r}{b_j}-\dist{r}{f_j}|$. 
    Now we have that 
    $$\dist{b_i}{b_j} \leq  \mbound + |\dist{r}{b_j}-\dist{r}{b_i}|.$$
   Let $ R_{ij}$ denote the $(b_i,b_j)$-subpath of $P$. From \Cref{clm:lb2} it follows that $|A \cap V(R_{ij})|\geq \mbound+2$. Now from \Cref{prp:antichain-length}, it follows that 
   $$\dist{b_i}{b_j} \geq \mbound+1+|\dist{r}{b_j}-\dist{r}{b_i}|$$ which is a contradiction.
 \end{subproof}

The following claim implies that if the subpath between some vertex $c\in V(P)$ and $b_i$, for some $i\in [t]$, contain many antichain vertices, then, $c$ must be far away from $Q'_i$.

\begin{claim}\label{clm:induced:2} 
    For an integer $i\in [t]$, let $c\in V(P)$ be a vertex such that the $(c,b_i)$-subpath of $P$ contains at least $2K+1$ antichain vertices of $A$. Then, for any vertex $d\in V(Q_i)$, $\dist{c}{d} \geq K$.
\end{claim}

\begin{subproof}
See \Cref{fig:clm:induced:2}. Suppose by contradiction that $\dist{c}{d} < K$ for some $d\in V(Q_i)$. Due to \Cref{prp:antichain-length}, $\dist{c}{b_i}\geq 2K + |\dist{r}{b_i} - \dist{r}{c}|$. Observe that $\dist{d}{b_i}\leq K$ would imply $\dist{c}{b_i} < 2K$, contradicting the above. Hence, there must exist a vertex $z\in V(Q_i)$, such that $\dist{r}{z}=\dist{r}{c}$ and $\dist{d}{z} < K$.
Hence, $\dist{b'_i}{c}\leq  \dist{b'_i}{z} + \dist{z}{d} + \dist{c}{d} < 2K + \dist{b'_i}{z} \leq  2K+\dist{r}{b'_i} - \dist{r}{z}$. Since $b'_i$ and $b_i$ are adjacent, $|\dist{r}{b_i} - \dist{r}{b'_i}|\leq 1$. Hence, $\dist{b_i}{c} \leq \dist{b'_i}{c}+1 < 2K+1+\dist{r}{b'_i} - \dist{r}{z} = 2K+|\dist{r}{b_i} - \dist{r}{c}| $, a contradiction.
\end{subproof}

\begin{figure}
    \centering
    \begin{subfigure}[b]{\textwidth}
    \centering
    \begin{tikzpicture}
			
			\draw[thick] (9.5,0) -- (0,0);
			\foreach \x/\y/\w/\z [count = \n] in
			{9/0/9/4, 7.5/0/7.5/4, 6/0/6/4, 3.5/0/3.5/4, 2/0/2/4, 0.5/0/0.5/4 }
			{
				\path [draw=black, thick,snake it] (\x,\y) -- (\w,\z);
				 \path [draw=black] (\w,\z) -- (4.5,6);
			}
			
			\foreach \x/\y [count = \n] in
			{9/0, 7.5/0, 6/0, 3.5/0, 2/0, 0.5/0 }
			{
				 \filldraw (\x, \y) circle (2pt);
			}

			\node at (4.75, 3) {\begin{tabular}{c}
					\ldots
			\end{tabular}};
			
			\definecolor{lightgray}{rgb}{0.83, 0.83, 0.83}
			
			\draw[rounded corners, fill=lightgray,opacity=0.25] (0, 3.8) -- (9.5, 3.8) -- (4.5,6.5) -- cycle;
			
			\draw[rounded corners, fill=lightgray,opacity=0.25] (0.5, 0.5) -- (1,-0.5) -- (0,-0.5) -- cycle;
			\draw[rounded corners, fill=lightgray,opacity=0.25] (2, 0.5) -- (2.5,-0.5) -- (1.5,-0.5) -- cycle;
			\draw[rounded corners, fill=lightgray,opacity=0.25] (3.5, 0.5) -- (4,-0.5) -- (3,-0.5) -- cycle;
			
			\draw[rounded corners, fill=lightgray,opacity=0.25] (6, 0.5) -- (6.5,-0.5) -- (5.5,-0.5) -- cycle;
			\draw[rounded corners, fill=lightgray,opacity=0.25] (7.5, 0.5) -- (8,-0.5) -- (7,-0.5) -- cycle;
			\draw[rounded corners, fill=lightgray,opacity=0.25] (9, 0.5) -- (9.5,-0.5) -- (8.5,-0.5) -- cycle;
					
			\node[left] at (4.5, 4.5) {$T$};
			\node[above] at (4.5, 6) {$r$};
			
			\foreach \x/\y [count = \n] in
			{ 0.5/2 , 2/2, 3.5/2   }
			{
				\node[left] at (\x, \y) {$Q'_{\n}$};
				\node[below] at (\x,2-\y) {$a_{\n}$};
				\node[below] at (\x,1.5-\y) {$Z_{\n}$};
			}
			
			\foreach \x/\y [count = \n] in
			{ 6/2, 7.5/2  }
			{
				\node[left] at (\x, \y) {$Q'_{t-\n}$};
					\node[below] at (\x,2-\y) {$a_{t-\n}$};
					\node[below] at (\x,1.5-\y) {$Z_{t-\n}$};
			}
			
			\foreach \x/\y [count = \n] in
			{ 9/2  }
			{
				\node[left] at (\x, \y) {$Q_{t}$};
				\node[below] at (\x,2-\y) {$a_{t}$};
				\node[below] at (\x,1.5-\y) {$Z_{t}$};
			}
			
			\node[below] at (5, 0) {$P$}; \node[left] at (0,0) {$u$}; \node[right] at (9.75,0) {$v$}; \node[above] at (4.5, 6) {$r$};  \filldraw (4.5, 6) circle (2pt);

	\end{tikzpicture} 
    \subcaption{The shaded triangular shapes indicate the branch sets $Z_1,Z_2,\ldots, Z_t$. The paths $Q'_1,Q'_2,\ldots,Q'_t$ and some and subpaths of $P$ (not indicated in this figure) are used as branch paths to construct a $K$-fat minor model of $U_t$.}
    \end{subfigure}
    
    \begin{subfigure}[b]{\textwidth}
    \centering
            \begin{tikzpicture}
                \draw[ultra thick, gray]   (0,0.75) -- (-0.5,0);
                \draw[ultra thick, gray]   (-1,0) -- (1,0);
                \filldraw[gray] (-1,-0.15) rectangle (-2,0.15);
                \filldraw[gray] (1,-0.15) rectangle (2,0.15);
                \foreach \x/\y [count = \n] in	{0/4, 0/0, -4/0, -3/0, 4/0, 0/0.75, -0.5/0, 3/0, 1/0, -1/0}
			{
				 \filldraw (\x, \y) circle (1.5pt);
			}

                \draw[thick]   (-4,0) -- (-1,0);
                \draw[thick]   (4,0) -- (1,0);
                
                \draw[thick, densely dotted] (0,4) -- (0,0); 
                \draw[thick, densely dotted] (0,4) -- (-3,0);
                \draw[thick, densely dotted] (0,4) -- (3,0);

                \node[above] at (0,4) {$r$};
                \node[below] at (0,0) {$a_i$};
                \node[below] at (-0.5,0) {$b_i$};
                \node[left] at (0,0.75) {$b'_i$};
                \node[below] at (-1.5,-0.2) {$D_{i-1}$};
                \node[below] at (1.5,-0.2) {$D_i$};
                \node[below] at (-3,0) {$b_{i-1}$};
                \node[below] at (3,0) {$b_{i+1}$};
            \end{tikzpicture}
            
        \subcaption{Construction of $Z_i$.  Vertices connected by the thick gray lines indicate $Z_i$. The dotted lines indicate $Q_i$. }
    \end{subfigure}
    \caption{Construction of branch sets and branch paths.}\label{fig:model-construct}
    
\end{figure}

We are now ready to construct the $K$-fat minor model of $U_t$, in order to provide a contradiction.

\medskip \noindent \textbf{Constructing the branch paths and branch sets.} \Cref{fig:model-construct} illustrates (on a high level) the construction of branch paths and branch sets of the $K$-fat minor model of $U_t$. Let $\mathcal{Q}=\{Q'_i\colon i\in [t]\}$. Let $R_i$ be the $(b_{i},b_{i+1})$-subpath of $P$ and let $D_i$ denote the maximum subpath of $R_i$ such that both components of the graph $E(R_i)\setminus E(D_i)$ contain $5K+2$ antichain vertices of $A$. Observe that $D_i$ always exists (due to \Cref{clm:lb2}). Let $\mathcal{D}=\{D_i\colon i\in [t-1]\} \}$.

Recall that, for $i\in [t]$, $c_i$ denotes the end-vertex of $Q'_i$ that is closest to $r$, and let $T_i$ denote the $(r,c_i)$-subpath of $Q_i$. Let $T$ denote the set $\displaystyle\bigcup\limits_{i\in [t]} V(T_i)$.

Clearly, the subgraph induced by $T$ is connected and $T$ will be the branch set corresponding to the universal vertex of $U_t$. Below, we construct the branch sets corresponding to the other vertices of $U_t$. We shall use the paths in $\mathcal{D}$ (defined above) to construct these branch sets.

 For each $i\in [t]$, let $X_i$ denote the $(a_i,b_i)$-subpath of $P$ and  $P_i=V(X_i)\cup \{b'_i\}$. Observe that the subgraph induced by $P_i$ is connected. 
 
 For $i\in [t-1]$, let $R_{i}$ denote the $(b_i,b_{i+1})$-subpath of $P$. Let $Y_1$ be the vertices of the connected component of $R_1 \setminus D_1$ that contains $b_1$. For $i\in [2,t-1]$, let $Y_i$ be the union of the vertices of the connected components of $R_{i-1}\setminus D_{i-1}$ and $R_{i+1}\setminus D_{i+1}$ that contain $b_i$. Finally, let $Y_t$ be the vertices of the connected component of $R_{t-1}\setminus D_{t-1}$ that contains $b_t$. For $i\in [t]$, let $Z_i=Y_i\cup P_i$. Note that the subgraph induced by $Z_i$ is connected and $Z_i\subseteq V(Q'_i)\cup V(P)$. To complete the proof, we need to show the following:

 \begin{enumerate}
     \item\label{it:p1} for each $i\in [t]$, $T\cap Q'_i\neq \emptyset$;
    \item\label{it:p2} $D_1\cap Z_1\neq \emptyset, D_{t-1}\cap Z_t\neq \emptyset$; and
    \item\label{it:p3} for $i\in [2,t-1], j\in [t-1], j-i\in \{0,1\},$  $Z_i\cap D_j\neq \emptyset$.
    
    \item\label{it:p4} for any two paths $X,Y\in \mathcal{Q}\cup \mathcal{D}$, $\dist{X}{Y}\geq K$.
    \item\label{it:p5} for each $i,j\in [t]$ with $i\neq j$, $\dist{Z_i}{Q'_j}\geq K$,
    \item\label{it:p6} for any $i\in [t]$, $\dist{T}{Z_i}\geq K$.
    \item\label{it:p7}  for any $i,j\in [t]$ with $i<j$, $\dist{Z_i}{Z_j}\geq K$,
    \item\label{it:p8}   for any $i\in [t-1]$, $\dist{T}{D_i} \geq K$, and
    \item\label{it:p9} for any $i\in [t-1], j\in [t]$ with $|j-i|\notin \{0,1\}$,
    we have $\dist{D_i}{Z_j} \geq K$.
 \end{enumerate}

\Cref{it:p1,it:p2,it:p3} follow from definitions. The next claim proves that the branch paths are far apart from each other (\Cref{it:p4}).

\begin{claim}\label{clm:far-away-0}
    For any two paths $X,Y\in \mathcal{Q}\cup \mathcal{D}$, $\dist{X}{Y}\geq K$.
\end{claim}
\begin{subproof}
If $X,Y\in \mathcal{Q}$, then we are done by the definition of $\mathcal{Q}$. Suppose $X$ is $Q'_i$ for some $i\in [t]$ and $Y=D_j$ for some $j\in [t-1]$. Since any vertex $y\in V(Y)$ is a vertex of $P$, the $(b_i,y)$-subpath of $P$, say $P'$, contains at least {$5K+2$} antichain vertices of $A$. Now, \Cref{clm:induced:2} implies the current claim. 

Finally, assume that $X=D_i$ for some $i\in [t-1]$. By definition, both components found after deleting $X$ from the $(b_i,b_{i+1})$-subpath have $5K+2$ antichain vertices of $A$, i.e., have length at least $5K+3$. So, if both $X,Y$ are in $\mathcal{D}$, they are part of $P$, and thus their distance is at least $5K+3\geq K$. Hence, also in this case the claim holds. 
\end{subproof}

The following claim proves \Cref{it:p5} (in a slightly stronger sense).

\begin{claim}\label{clm:strong}
    For each $i,j\in [t]$ with $i\neq j$, $\dist{Z_i}{Q_j}\geq K$.
\end{claim}
\begin{proof}
    Suppose that there exist $i,j\in [t]$ with $i\neq j$, and vertices $x\in V(Q'_j), y\in Z_i$ such that $\dist{x}{y}<K$. By definition of $Q'_j$, all its vertices are at distance at least $K$ from $Q_i$, so $y$ cannot be in $Q'_i$, which includes $b_i$ and $b'_i$ by \Cref{clm:3}. Hence, $y\in (V(P)\cap Z_i)\setminus\{b_i\}$. Observe that, the $(y,b_j)$-subpath of $P$ contains at least {$5K+2$} antichain vertices of $A$, which is a contradiction due to \Cref{clm:induced:2}.
\end{proof}

\medskip Now we prove \Cref{it:p6}.

\begin{claim}\label{clm:far-away-1}
For any $i\in [t]$, $\dist{T}{Z_i}\geq K$.
\end{claim}

\begin{subproof}
Suppose that there exists an integer $i\in [t]$, such that $\dist{x}{y}<K$ with $x\in T$ and $y\in Z_i$. The definition of $T$ implies that $x\in V(Q_j)$ for some $j\in [t]$. By \Cref{clm:strong}, it remains to consider the case when $i=j$. Recall that $c_i$ denotes the end-vertex of $Q'_i$ that is different from $b_i$. From the definition of $T$ (and $T_i$), it follows that $\dist{r}{x}\leq \dist{r}{c_i}$.

Let $d'_i$ denote the vertex of $Q_i$ which is adjacent to $c_i$ but does not lie in $Q'_i$. The definition of $Q'_i$ implies that there exists an integer $k\in [t]\setminus \{i\}$, such that $\dist{d'_i}{Q_k} < K$.
Let $d'_k$ denote a vertex of $Q_k$ with $\dist{d'_i}{d'_k}<K$.

\begin{figure}[t]
    \centering
    \begin{tabular}{cc}
        \begin{tikzpicture}
            \foreach \x/\y [count = \n] in {0/0,2/-4,-2/-4, 0/-1.55, 1.85/-2.25, -1.85/-2.25, 1.45/-1.55, 0.75/-0.75, -1.45/-1.55 }
			{
				\filldraw (\x, \y) circle (1.5pt);
			}
            
            \node[right] at (1.85,-2.25) {$d'_i$};
            \node[left] at (-1.45,-1.55) {$d'_k$};
            \node[above] at (0,-1.5) {$y$};
            \node[right] at (1.5,-1.5) {$y'$};
            \node[above] at (0.75,-0.75) {$x$};
            \node[left] at (-1.85,-2.25) {$z$};
            \node[left] at (-2,-4) {$b'_k$};
            \node[right] at (2,-4) {$b_i$};
            \node [above] at (0,0) {$r$};

            \draw (2,-4) ..controls (2, -2).. (0,0);
            \draw (-2,-4) ..controls (-2, -2).. (0,0);

            \draw[rounded corners] (2,-4) -- (0,-1.5) -- (-1,-4) -- (-2,-4);
            \foreach \x/\y/\w/\z [count=\n] in {0/-1.55/0.75/-0.75}
            {
                \draw[dashed] (\x,\y) -- (\w,\z);
            }
            \foreach \x/\y/\w/\z [count=\n] in {0/-1.55/0.75/-0.75, 1.85/-2.25/-1.45/-1.55}
            {
                \draw[dashed] (\x,\y) -- (\w,\z);
            }
            \foreach \x/\y/\w/\z [count=\n] in {0/-1.55/1.45/-1.55, 1.85/-2.25/-1.85/-2.25}
            {
                \draw[thick, dotted] (\x,\y) -- (\w,\z);
            }
        \end{tikzpicture} &         \begin{tikzpicture}
            \foreach \x/\y [count = \n] in {0/0,2/-4,-2/-4, 0/-2.55, 1.85/-2.25, 1.45/-1.55, -1/-1, -1.45/-1.55, -1.925/-2.55}
			{
				\filldraw (\x, \y) circle (1.5pt);
			}

            \foreach \x/\y/\w/\z [count=\n] in {-1.925/-2.55/0/-2.55, 1.45/-1.55/-1.45/-1.55}
            {
                \draw[thick, dotted] (\x,\y) -- (\w,\z);
            }
            \foreach \x/\y/\w/\z [count=\n] in {0/-2.55/1.45/-1.55,1.85/-2.25/-1/-1}
            {
                \draw[dashed] (\x,\y) -- (\w,\z);
            }

             \draw (2,-4) ..controls (2, -2).. (0,0);
            \draw (-2,-4) ..controls (-2, -2).. (0,0);
            \draw[rounded corners] (2,-4) -- (0,-2.5) -- (-1,-4) -- (-2,-4);

            \node[right] at (1.85,-2.25) {$d'_i$};
            \node[left] at (-1.45,-1.55) {$x'$};
            
            \node[left] at (-1,-1) {$d'_k$};
            \node[above] at (0,-2.55) {$y$};
            \node[right] at (1.5,-1.5) {$x$};
            \node[left] at (-1.925,-2.55) {$y'$};
            \node[left] at (-2,-4) {$b'_k$};
            \node[right] at (2,-4) {$b_i$};
            \node [above] at (0,0) {$r$};
        \end{tikzpicture} \\
      (a)   & (b)
    \end{tabular}
    \caption{Illustration of the proof of \Cref{clm:far-away-1}. Dashed lines indicate isometric paths of length at most $K$. (a) and (b) illustrate the cases when $\dist{r}{y}\leq \dist{r}{d'_i}$, and when $\dist{r}{y}> \dist{r}{d'_i}$, respectively. Dotted lines indicates that the two end-vertices of the line are at the same distance from $r$.}\label{fig:case-1}.
\end{figure}

First, consider the case when $\dist{r}{y} \leq \dist{r}{d'_i}$. See \Cref{fig:case-1}(a). Let $z$ be the vertex of $Q_k$ such that $\dist{r}{z}=\min\{ \dist{r}{b'_k}, \dist{r}{d'_i}\}$. We claim that, $\dist{d'_k}{z}\leq K$. When $z\neq b'_k$, since $\dist{d'_i}{d'_k}\leq K$, the vertex $z$ which is at same distance from $r$ as $d'_i$ must lie at a distance at most $K$ from $d'_k$. When $z=b_k$, $\dist{r}{b'_k} \geq \dist{r}{c}$ and since $d'_k$ lies in the $(z,r)$-isometric path, $\dist{d'_k}{z}\leq K$.

Let $y'$ be the vertex of $Q_i$ such that $\dist{r}{y} = \dist{r}{y'}$. Using similar arguments as above, it can be shown that $\dist{y'}{x}\leq K$.
Now observe the following:

\begin{align*}
        \dist{b_k}{y} & \leq 1 + \dist{b'_k}{d'_k} + \dist{d'_k}{d'_i} + \dist{d'_i}{x} + \dist{x}{y}\\
        & \leq 1 + \dist{b'_k}{z} + \dist{z}{d'_k}+ \dist{d'_k}{d'_i} + \dist{d'_i}{y'} + \dist{y'}{x}+ \dist{x}{y} \\
         & \leq 1 + (\dist{b'_k}{z} + \dist{d'_i}{y'}) + \dist{z}{d'_k}+ \dist{d'_k}{d'_i} + \dist{y'}{x}+ \dist{x}{y} 
    \end{align*}

Observe that $\dist{b'_k}{z} = \dist{r}{b'_k} - \dist{r}{z}$ and $\dist{d'_i}{y'}=\dist{r}{d'_i} - \dist{r}{y'}$. Now consider two cases. If $z\neq b'_k$, then $\dist{r}{z}=\dist{r}{d'_i}$ and therefore,

    \begin{align*}
        \dist{b_k}{y} & \leq 1 + (\dist{r}{b'_k} - \dist{r}{z}) + (\dist{r}{d'_i} - \dist{r}{y'}) + \dist{z}{d'_k}+ \dist{d'_k}{d'_i} + \dist{y'}{x}+ \dist{x}{y} \\
         & \leq 1 + (\dist{r}{b'_k} - \dist{r}{d'_i} + \dist{r}{d'_i} - \dist{r}{y'}) + \dist{z}{d'_k}+ \dist{d'_k}{d'_i} + \dist{y'}{x}+ \dist{x}{y} \\
        & < 1 + |\dist{r}{b'_k} - \dist{r}{y}| + 4K \\
         & \leq 4K + |\dist{r}{b_k} - \dist{r}{y}|
    \end{align*}

    If $z=b'_k$, then let $z'$ be the vertex on $Q_i$ such that $\dist{r}{b'_k}=\dist{r}{z'}$. From $\dist{d_i}{z}=\dist{d'_i}{b'_k}\leq K$, $\dist{r}{b'_k}\leq \dist{r}{d'_i}$, and $\dist{d'_k}{d'_i}\leq K$, it follows that $\dist{z'}{d'_i}\leq K$. Hence,  

        \begin{align*}
        \dist{b_k}{y} & \leq 1 + \dist{z}{d'_k} + \dist{d'_k}{d'_i} + \dist{z'}{d'_i} + \dist{z'}{y'} + \dist{y'}{x}+ \dist{x}{y} \\
         & \leq 1 + (\dist{r}{z'} - \dist{r}{y'}) +  \dist{b'_k}{d'_k}+ \dist{d'_k}{d'_i} + \dist{y'}{x}+ \dist{x}{y} + \dist{z'}{y'} \\
        & < 1 + \dist{r}{b'_k} - \dist{r}{y} + 5K \\
         & \leq 5K + |\dist{r}{b_k} - \dist{r}{y}|
    \end{align*}

    The above arguments imply that $\dist{b_k}{y}\leq 5K + |\dist{r}{b_k} - \dist{r}{y}|$.
Now, observe that the $(b_k,y)$-subpath (say $P'$) of $P$ contains $D_k$. Therefore, $P'$ contains at least one component of the graph obtained by removing the edges of $D_k$ fromn the $(b_k,b_{k+1})$-subpath of $P$. The definition of $D_k$ now implies that the $(b_k,y)$-subpath of $P$ has at least $5K+2$ antichain vertices of $A$. Therefore, \Cref{prp:antichain-length} implies that $\dist{y}{b_k} \geq 5K+1+|\dist{r}{b_k} - \dist{r}{y}|$, a contradiction.

Now, consider the case where $\dist{r}{y} > \dist{r}{d'_i}$: see \Cref{fig:case-1}(b). Let $y'$ be the vertex of $Q_k$ such that $\dist{r}{y'}=\min\{ \dist{r}{b'_k}, \dist{r}{y}\}$. Let $x'$ be the vertex of $Q_k$ such that $\dist{r}{x} = \dist{r}{x'}$. Since $\dist{x}{y}< K$, $\dist{r}{x}=\dist{r}{x'}$, and $x'$ lies on a $(r,y')$-isometric path, we have $\dist{x'}{y'}< K$. 
Since $\dist{r}{x}< \dist{r}{d'_i}<\dist{r}{y}$ and $\dist{x}{y} < K$, we have $\dist{x}{d'_i}< K$. 
Now observe the following:

\begin{align*}
        \dist{b_k}{y} & \leq  1 + \dist{b'_k}{y'} + \dist{x'}{y'} + \dist{d'_k}{x'} + \dist{d'_k}{d'_i} + \dist{d'_i}{x} + \dist{x}{y}\\
         & \leq 1 + |\dist{r}{b'_k} - \dist{r}{y}| + \dist{x'}{y'} + \dist{d'_k}{x'} + \dist{d'_k}{d'_i} + \dist{d'_i}{x} + \dist{x}{y} \\
        & \leq 1 + |\dist{r}{b'_k} - \dist{r}{y}| + 5K \\
         & \leq 5K + |\dist{r}{b_k} - \dist{r}{y}|
    \end{align*}

Now, observe again that the $(b_k,y)$-subpath of $P$ has at least $5K+2$ antichain vertices of $A$. Therefore, \Cref{prp:antichain-length} implies that $\dist{y}{b_k} \geq 5K+1+|\dist{r}{b_k} - \dist{r}{y}|$. But this is a contradiction.

\end{subproof}

Now we prove \Cref{it:p7}. 
 \begin{claim}\label{clm:far-away-2}
     For any $i,j\in [t]$ with $i<j$, $\dist{Z_i}{Z_j}\geq K$.
 \end{claim}
\begin{subproof}
The definitions of $D_i$ and $D_j$ imply that both of them contain at least $2K-2$ antichain vertices each, implying the length of both $D_i$ and $D_j$ are at least {$2K-3$}. Since $K\geq 5$, we have $2K-3>K+1$. Now, the claim follows from the definitions of $Z_i$ and $Z_j$.
\end{subproof}

 Now we show that the constructed branch sets are sufficiently far away from the branch paths, as required. First, we prove \Cref{it:p8}.

\begin{claim}\label{clm:far-away-3}
    For any $i\in [t-1]$, $\dist{T}{D_i} \geq K$.
\end{claim}
\begin{subproof}
Suppose that there is a vertex $x\in D_i$ and $y\in T$ such that $\dist{x}{y} < K$. The definition of $T$ implies that $y\in V(Q_j)$ for some $j\in [t]$. Since the $(b_j,x)$-subpath of $P$ contains at least $5K+2$ antichain vertices of $A$, we have a contradiction due to \Cref{clm:induced:2}. 
\end{subproof}

Using similar arguments as \Cref{clm:far-away-3}, we can prove \Cref{it:p9}.

\begin{claim}\label{clm:far-away-4}
For any $i\in [t-1], j\in [t]$ with $|j-i|\notin \{0,1\}$, we have $\dist{D_i}{Z_j} \geq K$.
\end{claim}

The proof of \Cref{lem:house-ipac} is now complete.

\subsection{Monoholed graphs}\label{sec:mono}

Now, we shall prove \Cref{thm:mono}. We need the definitions of \emph{Truemper configurations}~\cite{vuskovic2013world}. A \emph{theta} is a graph made of three internally vertex-disjoint induced paths $P_1 = a\ldots b$, $P_2 = a\ldots b$, $P_3 = a\ldots b$ of lengths at least~2, and such that no edges exist between the paths except the three edges incident to $a$ and the three edges incident to $b$. A \emph{pyramid} is a graph made of three induced paths $P_1 = a\ldots b_1$, $P_2 = a\ldots b_2$, $P_3 = a\ldots b_3$, two of which have length at least~$2$, that are vertex-disjoint (except at $a$), such that $b_1 b_2 b_3$ is a triangle, and no edges exist between the paths except those of the triangle and the three edges incident to $a$. A \emph{prism} is a graph made of three vertex-disjoint induced paths $P_1 = a_1\ldots b_1$, $P_2 = a_2\ldots b_2$, $P_3 = a_3\ldots b_3$ of lengths at least $1$, such that $a_1 a_2 a_3$ and $b_1 b_2b_3$ are triangles and no edges exist between the paths except those of the two triangles. 

A \emph{$3$-path configuration} is a theta, prism or a pyramid. A $3$-path configuration is \emph{balanced} if the lengths of all the paths in the respective definitions are the same; otherwise, it is \emph{unbalanced}. The \emph{length} of a balanced 3-path configuration is the length of the paths in the respective definitions. Notice that all $3$-path configurations in a monoholed graph must be balanced and of the same length. Now, we will be done by proving the following lemma.

\begin{lemma}
    Let $G$ be a graph such that all $3$-path configurations contained in $G$ are balanced and of the same length. Then $G$ does not contain a $10$-fat minor model of $U_4$. 
\end{lemma}

\begin{proof}

\newcommand{\connector}[2]{C_{#1}^{#2}}
\newcommand{\witness}[1]{W_{#1}}

\begin{figure}
	\centering
	\begin{tabular}{cc}
		\begin{tikzpicture}
			\foreach \x/\y [count = \n] in {0/0,1/0,2/0,3/0,1.5/1}
			{
				\filldraw (\x, \y) circle (2pt);
				
			}
			
			\node[above] at (1.5,1) {$a_0$};
			\foreach \x/\y [count = \n] in {0/0, 1/0,2/0,3/0}
			{
				\node[below] at (\x, \y) {$a_{\n}$};
			}
			
			\foreach \x/\y/\w/\z [count = \n] in {0/0/1/0,1/0/2/0,2/0/3/0}
			{
				\draw (\x,\y) -- (\w,\z);
				\draw (1.5,1) -- (\x,\y);
			}
			
			\draw (1.5,1) -- (3,0);
		\end{tikzpicture}& \begin{tikzpicture}[scale=0.85]
			\foreach \x/\y [count = \n] in {0/0, 3/0,6/0,9/0}
			{
				\draw[dashed] (\x,\y) ellipse (1cm and 0.5cm);
				\node[below] at (\x,\y-0.5) {\scriptsize $B_{\n}$};
			}
			
			\foreach \x/\y/\w/\z [count = \n] in {4/0.85/0/2, 5/0.85/0/3, 1.5/2.25/0/1,1.5/0/1/2,4.5/0/2/3,7.5/0/3/4, 7.5/2.25/0/4}
			{
				\node[above] at (\x,\y) {\scriptsize $\connector{\w}{\z}$};
			}
			
			\draw[dashed] (4.5,2) ellipse (1cm and 0.5cm);
			\node[above] at (4.5,2.5) {$B_{0}$};
			
			\draw [dashed, fill=gray,opacity=0.25,rounded corners] (2.75,0.35) -- (4,1.75) -- (4.25,1.75) -- (3,0.35) -- cycle;
			
			\draw [dashed, fill=gray,opacity=0.25,rounded corners] (6.25,0.35) -- (5,1.75) -- (4.75,1.75) -- (6,0.35) -- cycle;
			
			\draw [fill=gray,opacity=0.25,rounded corners] (-0.5,0.25)  ..  controls (-0.5,2.25) .. (3.8,2.25) -- (3.8,2) .. controls (-0.25,2) .. (-0.25,0.25)  -- cycle;
			
			\draw [fill=gray,opacity=0.25,rounded corners] (9.5,0.25)  ..  controls (9.5,2.25) .. (5.25,2.25) -- (5.25,2) .. controls (9.25,2) .. (9.25,0.25)  -- cycle;
			
			\foreach \x/\y [count = \n] in {1.5/0, 4.5/0,7.5/0}
			{
				\draw[dashed,fill=gray,opacity=0.25] (\x,\y) ellipse (0.75cm and 0.15cm);
			}

            \draw[line width=0.3mm] (0.5,0) -- (8.5,0);
            \draw [line width=0.3mm,rounded corners] (0.5,0) -- (0,0) -- (-0.375,0.2) ..  controls (-0.375,2.125) .. (3.8,2.125) -- (4.5,2.125);

            \draw [line width=0.3mm, rounded corners] (8.5,0) -- (9,0) -- (9.375,0.2) ..  controls (9.375,2.125) .. (5,2.125) -- (4.5,2.125);

            \draw [line width=0.12mm, rounded corners] (5,2.125) -- (4,1.625) -- (2.9,0.4) -- (2.9,0);

            \draw [line width=0.12mm, rounded corners] (4,2.125) -- (5,1.625) -- (6.1,0.4) -- (6.1,0);

            \foreach \x/\y [count = \n] in {6.1/0,4/2.125,5/2.125}
			{
				\filldraw (\x, \y) circle (2pt);
				
			}
            \node [above] at (4.15,2.05) {\scriptsize $y$};
            \node [below] at (6.1,0) {\scriptsize $x$};
            \node [below] at (5,2.125) {\scriptsize $z_3$};
            
            \foreach \x/\y [count = \n] in {-1/1,2.75/1,10.5/1}
			{
				\node[left] at (\x, \y) {\scriptsize$ P_{\n}$};
				
			}
            \draw[->] (-1.1,1) -- (-0.375,1); 
            \draw[->] (2.65,1) -- (3.375,1);
            
            \draw[->] (10,1) -- (9.4,1);
		\end{tikzpicture}\\
		(a) & (b)
	\end{tabular}
	\caption{(a) The graph $U_4$. (b) Illustration of the proof of \Cref{thm:mono}. The  thick, black, closed curve indicates the induced cycle $D$ (in \Cref{clm:long-cycle}). }\label{fig:mono-1}
\end{figure} 

Assume for contradiction that $G$ contains a $K$-fat minor model of $U_4$ with $K\geq 10$. Let the vertices of $U_4$ be denoted by $a_0,a_1,a_2,a_3,a_4$ as shown in \Cref{fig:mono-1}(a). Unless stated explicitly, all addition operations in this proof will be taken modulo $5$. Let $B_i$ denote the branch set corresponding to the vertex $a_i$ and $\connector{i}{j}$ denote the branch path corresponding to the edge between $a_i$ and $a_j$. See \Cref{fig:mono-1}(b) for illustrations.
For the sake of readability, we shall use the same notations to indicate the subgraphs induced by the branch sets and branch paths. Let $Z=\{(0,1), (0,2), (0,3), (0,4), (1,2), (2,3), (3,4)\}$. 
A \emph{witness} of a branch path $\connector{i}{j}, (i,j)\in Z$ is an induced path $P$ such that $V(P)\subseteq \connector{i}{j}$, one end-vertex of $P$ lies in $B_i$, and the other end-vertex of $P$ lies in $B_j$. 
For $0\leq i\leq 4$, let $\witness{i}$ be a witness of $\connector{i}{i+1}$ of shortest possible length, and let $u_i\in B_i,v_{i+1}\in B_{i+1}$ be the end-vertices of $\witness{i}$. For $0\leq i\leq 4$, let $X_i$ denote a path of shortest possible length between $u_i$ and $v_i$ such that $V(X_i)\subseteq B_i$.

Let $H$ be the subgraph of $G$ induced by $\displaystyle\bigcup\limits_{i=0}^4 V(X_i)\cup V(W_i)$.
\begin{claim}\label{clm:long-cycle}
    The graph $H$ contains an induced cycle $D$ such that for each $0\leq i\leq 4$, both $V(B_i)\cap V(D)$ and $V(\connector{i}{i+1})\cap V(D)$ induce paths of lengths at least $K-2$.

\end{claim}
\begin{subproof}
The minimality of $\witness{i}$ implies that only the end-vertices of $\witness{i}$ or their neighbors (in $\witness{i}$) can be adjacent to some vertices of $X_i$ or $X_{i+1}$. Recall that for each $0\leq i\leq 4$, $\witness{i}$ has length at least $K$ since it connects two branch sets. Hence, there must exist a subpath of $\witness{i}$, say $\witness{i}'$, of length at least $K-2$,
such that no internal vertex of $\witness{i}'$ is adjacent to a vertex of $X_{i}$
or $X_{i+1}$. Similarly, since the distance between two witnesses (corresponding to two branch paths) is at least $K$, observe that there must exist a subpath of $X_i$, say $X'_i$, of length at least $K-2$, none of whose internal vertex is adjacent to the vertices of $\witness{i-1}$ or $\witness{i}$.
Now, using $\witness{i}'$ and $X'_i$, $0\leq i\leq 4$, it is possible to construct an induced cycle $D$ satisfying the claim.
\end{subproof}

Let $D$ be the induced cycle whose existence was proved in \Cref{clm:long-cycle}. For $i\in \{0,1,2,3,4\}$, let $D_i=B_i\cap D$. Let $Q$ be an induced path of the shortest possible length between a vertex of $D_2$ and $D_0$ such that $V(Q)\subseteq B_0\cup \connector{0}{2}\cup B_2$. Since $B_0, \connector{0}{2},B_2$ are connected induced subgraphs, $Q$ always exists. 
Notice that $Q$ and $D$ together create a $3$-path configuration (which is balanced since $G$ is monoholed), say 
$P_1$, $P_2$, $P_3$, such that $V(P_1)\cap V(\witness{0})\neq \emptyset, V(P_2) \cap \connector{0}{2} \neq \emptyset$ and $V(P_3)\cap V(\witness{4})\neq \emptyset$. See \Cref{fig:mono-1}(b) for an illustration.
For $i\in [3]$, let $z_i$ (resp. $z'_i)$ denote the end-vertex of $P_i$ which is adjacent to some vertices of $D_0$ (resp. $D_2)$. Let $K'$ be the length of the three paths. Observe that the vertices of the paths $P_2$ and $P_1$ (resp. $P_2$ and $P_3$) are contained in 
a cycle $F$ (resp. $F'$) of order at most $2K'+2$ which also contains $D_1$ (resp. $D_3$). ote that the vertices of $F$ are either contained in $P_1$ or $P_2$. Similarly, the vertices of $F'$ are contained in either $P_2$ or $P_3$. Moreover, $D_3$ is a subpath of $P_3$.

Now, consider an induced path $R$ of the shortest possible length between a vertex of $D_3$ and $D_0$ such that $V(R) \subseteq B_0 \cup \connector{0}{3} \cup B_3$. Let $x$ and $y$ denote the end-vertices of $R$ that lie on $D_3$ and $D_0$, respectively. Let $x',y'$ be the neighbours of $x$ and $y$ on $R$, respectively and $R'$ be the $(x',y')$-subpath of $R$. The minimality of $R$ implies that no internal vertex of $R'$ is adjacent to any vertex of $D_0$ and $D_3$. Furthermore, we prove the following claim.

\begin{claim}
    No vertex of $R'$ is adjacent to any vertex of $F'$.
\end{claim}
\begin{subproof}
Suppose for a contradiction that there is a vertex of $R'$ which is adjacent to some vertex of $F'$. Let $w$ be the vertex of $R'$, which is closest to $x'$ and is also adjacent to some vertex $w'$ of $F'$. Note that $w'\in V(P_2)$ or $w'\in V(P_3)$. 

Suppose first that $w'\in V(P_3)$. Observe that $x$ is a vertex of $D_3$, which implies that $x\in V(P_3)$. Let $P'$ denote the path created by the vertices of the $(w,x')$-subpath of $R$ along with $w'$ and $x$. Observe that both endpoints of $P'$ lie in $P_3$. Therefore, the $(w',x)$-subpath of $P_3$ has strictly smaller length than $P_3$. Hence,
$F'$ together with $P'$ must create a $3$-path configuration, which is either unbalanced or has length strictly less than $K'$, a contradiction. For similar reasons, $w'$ cannot be $z_2$. 

Now, consider the case where $w'$ is an internal vertex of $P_2$. Let $U$ be the maximal subpath of $P_3$ whose one end-vertex is $z'_3$, and no internal vertex of $U$ is a neighbor of $x'$. Observe that $U$ together with the $(w,x')$-subpath of $R'$ creates an induced path $U'$ whose (only) end-vertices are adjacent to the vertices of $F$ (the induced cycle created by $P_2$ and $P_1$). In particular, both end-vertices of $U'$ lies on $P_2$ or is adjacent to a vertex of $P_2$.  Since $w'$ is an internal vertex of $P_2$, $F$ together with $U'$ must create a $3$-path configuration $Q$ containing a path which is a proper subpath of $P_2$. Hence, either $Q$ is unbalanced or has a length strictly less than $K'$, a contradiction.
\end{subproof}

Let $S$ be the $(y,z_3)$-subpath of $D_0$ and let $S'$ be the subpath of $S$ of minimum possible length such that one end-vertex of $S'$ is adjacent to $y'$ and the other end-vertex is adjacent to $z_2$. Let $R''$ denote the path induced by $V(R')\cup V(S')$. Then, $F'$ along with $R''$ contains a $3$-path configuration $Q$ which is either unbalanced or has length strictly greater than $K'$, a contradiction.
\end{proof}

\subsection{Even-hole free graphs}\label{sec:even}
\newcommand{\Basic}[1]{B_{#1}}
\newcommand{\EvenH}[1]{G_{#1}}

\begin{figure}[t]
\centering
\begin{tabular}{cc}
	\begin{tikzpicture}
		 \foreach \x/\y [count = \n] in {0/0,0/0.5,0/1,0/1.5,0/2,-0.5/-0.5,0.5/-0.5}
		 {
		 	\filldraw  (\x,\y) circle (2pt);
		 }
		 
		 \draw [thick] (0,2) -- (0,0) -- (-0.5,-0.5) -- (0.5,-0.5) -- (0,0);
	\end{tikzpicture} &\begin{tikzpicture}[scale=0.85]
	\foreach \z [count=\m] in {0,3,6,9}
	{
	
		\foreach \x/\y [count = \n] in {\z/0,\z/0.5,\z/1,\z/1.5,\z/2,\z-0.5/-0.5,\z+0.5/-0.5}
		{
			\filldraw  (\x,\y) circle (2pt);
			
		}
	
		\node[above] at (\z,2) {$r_{\m}$};
		\node[left] at (\z,0) {$a_{\m}$}; \node[left] at (\z-0.5,-0.5) {$c_{\m}$};\node[right] at (\z+0.5,-0.5) {$b_{\m}$};

		\draw [thick] (\z,2) -- (\z,0) -- (\z-0.5,-0.5) -- (\z+0.5,-0.5) -- (\z,0) -- cycle;

	}
	
	\foreach \z [count=\m] in {0,3,6}
	{
			\foreach \x [count = \n] in {\z+0.75,\z+1.5,\z+2.25}
		{
			\filldraw  (\x,2) circle (2pt);
		}
		
		\foreach \x [count = \n] in {\z+0.75,\z+1.25,\z+1.75,\z+2.25}
		{
			\filldraw  (\x,-1) circle (2pt);
		}
		
		\draw[thick] (\z,2) -- (\z+3,2);
		
		\draw[thick] (\z+0.5,-0.5) -- (\z+0.75,-1) -- (\z+2.25,-1) -- (\z+2.5,-0.5);
	}
	\end{tikzpicture}  \\
	(a) & (b)
\end{tabular}
\label{fig:evenholefree}
\caption{Proof of \Cref{thm:even}. (a) The graph $B_4$. (b) The graph $G_4$.}
\end{figure}

In this section, we show \Cref{thm:even} that for each integer $k\geq 1$, there exists an even-hole free graph $Z_k$ with maximum degree at most $3$,
such that $\ipco{Z_k}\geq k$. It is enough to prove the statement for each integer $k\geq 4$ (and setting $Z_i=Z_4$ for $i=1,2,3$). For an integer $k\geq 4$, let $\Basic{k}$ be the graph obtained by taking a triangle $T$ and attaching to one of its vertices a path of length $k$ (see \Cref{fig:evenholefree}. 
We construct $G_k$ as follows.
\begin{itemize}
    \item Let $F_1, F_2, \ldots, F_k $ be disjoint copies of $\Basic{k}$. Let $a_i,b_i,c_i$ be the three vertices of the unique triangle of $F_i$, where the degree of $a_i$ is~3, and the only vertex of degree~1 in $F_i$ is denoted by $r_i$. 

    \item Connect $r_i$ and $r_{i-1}$ using a path $P_i$ of length $k$. Connect $c_i$ and $b_{i-1}$ using a path $Q_i$ of length $k+1$. The resulting graph is $G_k$.
\end{itemize}

Let $R=\{r_i\colon i\in [k]\}$. We prove the following lemma.
\begin{lemma}\label{lem:even-hole-free}
    The graph $\EvenH{k}$ does not contain any even hole.
\end{lemma}
\begin{proof}
Let $C$ be an induced cycle of length at least $4$. It is clear from the construction that $C$ will contain at least one vertex from $R$. Moreover, if $r_i, r_t\in V(C)$ for some $i<t$, then for any $i<j<t$, $r_j\in V(C)$. Let $i$ (resp. $t$) be the minimum (resp. maximum) index such that $r_i\in V(C)$ (resp. $r_t\in V(C)$). Hence, for each $i<j\leq t$, the path $P_j$ is also part of $C$. It is also easy to verify that for each $i<j\leq t$, the path $Q_j$ is also part of $C$. Let $X_i$ (resp. $X_t$) be the path between $r_i$ and $b_i$ (resp. $r_t$ and $c_t$). Observe that $$V(C)=V(X)\cup V(X') \displaystyle\bigcup\limits_{i<j\leq t} \left(V(P_j)\cup V(Q_j) \right) \displaystyle\bigcup\limits_{i<j<t} \{b_j,c_j\} $$ Hence the length of $C$ is $(t-i)k+(t-i)(k+1) + 2(k+1) + (t-i)-1 =(t-i+1)(2k+2)-1$, which is always odd. 
\end{proof}

\begin{lemma}\label{lem:even-hole-ipco}
Let $r=r_k, G=G_k$. Then $\ipcor{r}{G}\geq k$. 
\end{lemma}
\begin{proof}
 Let $Q$ be the induced path between $b_k$ and $c_1$ induced by the vertices of $\left(\displaystyle\bigcup\limits_{i=2}^k V(Q_i)\right) \cup \left(\displaystyle\bigcup\limits_{i=1}^k \{b_j,c_j\}\right)$. The length of $Q$ is $k^2+k-1$. 
    We claim that $Q$ is an isometric path between $c_1$ and $b_k$. To prove this, consider a path $Q'$ between $c_1$ and $b_k$ that is different from $Q$. Notice that $Q'$ must contain at least one vertex from $R$ and let $i$ be the minimum index such that $r_i\in R$. Observe that $\dist{b_k}{r_i}\geq (k-i+1)k+1$ and $\dist{r_i}{c_1}\geq (i+1)k+1$. Therefore, the length of $Q'$ is at least $(k+2)k+2$, which is larger than the length of $Q$. Therefore, $Q$ is isometric and contains all edges in $\{b_ic_i\colon i\in [k]\}$.

    Now, observe that for each $i\in [k]$, $\dist{r}{b_i}=\dist{r}{c_i}$. Therefore, two distinct $r$-rooted isometric paths are needed to cover $b_i$ and $c_i$, for $i\in [k]$. Moreover, it is clear from the construction that for $i<j\in [k]$, there is no directed path from $b_j$ to $b_i$ nor from $b_i$ to $b_j$ in $\overrightarrow{G_r}$. The above arguments imply that $|\anticp{r}{Q}|\geq k$ and therefore $\ipco{G,r}\geq k$.     
\end{proof}

\noindent \textbf{Completion of the proof of \Cref{thm:even}.} Consider two disjoint copies $X, Y$ of the graph $\EvenH{k}$. Let $r_1$ and $r_2$ denote the vertices labeled $r_k$ in $X$ and $Y$, respectively. Now, add an edge between the vertices $r_1$ and $r_2$. Let $Z_k$ be the resulting graph. Observe that $Z_k$ is even-hole free, due to \Cref{lem:even-hole-free} and the fact that adding an edge between two disjoint even-hole free graphs results in an even-hole free graph. Let $v$ be any vertex of $Z_k$ and without loss of generality, $v\in V(X)$. Observe that, for any vertex $u\in V(Y)$, the $(u,v)$-isometric path contains the vertex $r_2$. Hence, for any isometric path $P$ that lies in $Y$, a set of $v$-rooted isometric paths that cover $P$, is also a set of $r_2$-rooted isometric paths. Hence, $\ipco{Z_k,v}\geq \ipco{Y, r_2} \geq k$.

\begin{remark}\label{rem:quasi-isometry}
The graphs constructed to prove \Cref{thm:even} can also be used to show that the property of bounded (strong) isometric path complexity is not preserved under quasi-isometry. Indeed, consider the graph $G_k$ and the $(c_1,b_k)$-isometric path that does not go through $r_1$. By appropriately subdividing some edges of this path, it is possible to obtain a graph $H$ which is quasi-isometric to $G$ but has constant (strong) isometric path complexity.
\end{remark}
\vspace{-0.25cm}

\section{Graph operations}\label{sec:graph-op}

In \Cref{sec:pow}, \Cref{sec:line} and \Cref{sec:clique} we prove \Cref{thm:line-power}\ref{it:pow}, \Cref{thm:line-power}\ref{it:line} and \Cref{thm:clique-sum}, respectively.
\subsection{Graph powers}\label{sec:pow}

Let $G$ be a graph with $\ipco{G}\leq k$ and $r\geq 2$ be a fixed integer. Let $H=G^r$. Now color the edges of $H$ as follows: an edge $e\in E(H)$ is \emph{black} if $e\in E(G)$; otherwise $e$ is \emph{red}. An isometric path $P$ of $H$ is \emph{red} if all edges of $P$ are red. 
{We will use the following simple observation.}

\begin{observation}\label{obs:iso:lb}
	Let $P$ be a red $(u,v)$-isometric path of length $\ell \geq 1$ in $H$. Then the distance between $u$ and $v$ in $G$ is at least $r(\ell-1)+1$.
\end{observation}
We prove the following lemma.

\begin{lemma}\label{lem:pow-2}
	The vertices of any $v$-rooted isometric path $Q$ of $G$ can be covered by $r$ many $v$-rooted isometric paths in $H$.
\end{lemma}
\begin{proof}
	Let the vertices of $Q$ be ordered as $v_1=v, v_2, \ldots, v_\ell$ as they are encountered while traversing $Q$ from $v$. Observe that for $i\in [r]$, the vertices in the set $P_i=\{v_j \colon j=i+1+xr,x\in \mathbb{N}, i+1+xr\leq \ell\}$ induce an isometric path in $H$. Since $v$ and $v_i$ with $i\in [2,r]$ are adjacent {in $H$}, we have that $Q_i=P_i\cup\{v\}$ is also an isometric path in $H$. Now, the set $\mathcal{Q}=\{Q_i\colon i\in [r]\}$ covers all the vertices of $Q$.
\end{proof}

\begin{lemma}\label{lem:black}
	Any isometric path of $H$ contains at most one black edge.
\end{lemma}
\begin{proof}
	Assume for contradiction that there is a path $P$ in $H$ that has at least two black edges $u_1v_1$ and $u_2v_2$. Let the vertices $u_1,v_1, u_2,v_2$ appear in that order while traversing $P$ from one end-vertex to the other. Also, assume that the subpath $P'$ of $P$ between $u_1$ and $v_2$ does not contain any other black edge. Clearly, $u_2\neq v_1$. (Otherwise, the distance between $u_1$ and $v_2$ is at most~2 in $G$, which implies that $u_1,v_2$ would be adjacent in $H$, contradicting that $P$ is isometric.) Hence, $P'$ has at least one edge, and $P'$ is a red isometric path in $H$ with $v_1$ and $u_2$ as its end-vertices. Let $x_1=v_1, x_2, \ldots, x_\ell=u_2$ be the vertices of $P'$ as they are encountered while traversing $P'$ from $v_1$ to $u_2$. For each $i\in [\ell-1]$, let $P_i$ denote an $(x_i,x_{i+1})$-isometric path of length at most $r$ in $G$. (Since $x_ix_{i+1}$ is a red edge, such a path always exists.) For each $i\in [\ell-1]$, let $y_i$ denote the vertex of $P_i$ that is adjacent to $x_i$. Observe that, for each $i\in [\ell-2]$, $y_i$ is adjacent to $y_{i+1}$ in $H$. Moreover, $u_1$ is adjacent to $y_1$ in $H$ and $y_{\ell-1}$ is adjacent to $v_2$ in $H$. Hence, $u_1,y_1,y_2,\ldots,y_{\ell-1},v_2$ induces a path of length $\ell$ in $H$. But this contradicts the fact that $u_1,v_1=x_1,\ldots,x_{\ell}=v_2,u_2$ induces an isometric subpath of $P$ with length $\ell+1$. 
\end{proof}

\begin{lemma}\label{lem:isometric-cover-power}
	The vertices of any red isometric path of $H$ can be covered by $2r-1$ many isometric paths of $G$.
\end{lemma}

\begin{proof}
	Let $P$ be a red isometric path of $H$ with $|E(P)| = \ell$. Let $u_1,u_2,\ldots,u_{\ell+1}$ be the vertices of $P$ ordered as they are traversed from one end-vertex of $P$ to the other. Observe that the distance in $G$ between consecutive vertices of $P$ is at most $r$. For each $i\in [\ell]$, let $e_i=u_{i}u_{i+1}$ and $Q_i$ denote an $(u_i,u_{i+1})$-isometric path in $G$. Finally, let $Q$ be the path in $G$ obtained by concatenating $Q_1,Q_2,\ldots,Q_{\ell}$. Observe that $\forall i\in [\ell], |E(Q_i)| \leq r$ and therefore $|E(Q)|\leq \ell\cdot r$

	\medskip	Let $\mathcal{Z}$ be a maximal collection of edge-disjoint non-isometric subpaths of $Q$, and $\overline{\mathcal{Z}}$ is the set of connected components in the graph obtained by removing all paths in $\mathcal{Z}$ from $Q$. 
	Let $\mathcal{Y}=\mathcal{Z}\cup \overline{\mathcal{Z}}$. We shall also use the following. 
	
	\begin{claim}\label{clm:iso}
		Every path $X\in \overline{\mathcal{Z}}$ is an isometric path in $G$.
	\end{claim}
	
	The definitions imply that any two paths in $\mathcal{Y}$ are edge disjoint. Hence, the paths in $\mathcal{Y}$ can be ordered as $Y_1,Y_2,\ldots,Y_p$ such that for $1\leq i <j\leq p$ all edges of $Y_i$ appear before those of $Y_j$ while traversing from one end-vertex of $Q$ to the other.  
	Hence, one end-vertex of $Y_{p}$ must be $u_{\ell+1}$. Let the two end-vertices of $Y_i$ be denoted as $a_i$ and $b_i$. Note that $a_{i+1}=b_i$ and without loss of generality we can assume that $a_1=u_1$, and $b_p=u_{\ell+1}$. Note that $\displaystyle\sum\limits_{i=1}^p |E(Y_i)|=|E(Q)| \leq \ell\cdot r$.
	Using the above notations, we prove the following claim.
	
	\begin{claim}\label{clm:upper_bound}
		$|\mathcal{Y}|\leq 2r-1.$
	\end{claim} 
	\begin{subproof}
		Observe that $\left| \overline{\mathcal{Z}} \right| \leq \left| \mathcal{Z} \right|+1 $. Hence proving $|\mathcal{Z}|\leq r-1$ is enough to prove \Cref{clm:upper_bound}. Assume for contradiction that $|\mathcal{Z}|\geq r$. 
		Let $X_i$ denote an $(a_i,b_i)$-isometric path in $G$ and let $Q'$ be the path between $u_1$ and $u_{\ell+1}$ obtained by concatenating $X_1,X_2,\ldots,X_p$. Observe that for each $i\in [p]$, $Y_i\in \overline{\mathcal{Z}}$, then due to \Cref{clm:iso}, $|E(X_i)|= |E(Y_i)|$. On the other hand, if for each $i\in [p]$, $Y_i\in \mathcal{Z}$, then $|E(X_i)|\leq |E(Y_i)|-1$. Hence,
		\begin{align*}
			|E(Q')| & = \displaystyle\sum\limits_{Y_i\in \overline{\mathcal{Z}}} |E(X_i)| +  \displaystyle\sum\limits_{Y_i\in \mathcal{Z}} |E(X_i)| \\
			& \leq \displaystyle\sum\limits_{Y_i\in \overline{\mathcal{Z}}} |E(Y_i)| + \displaystyle\sum\limits_{Y_i\in \mathcal{Z}} \left(|E(Y_i)|-1\right) \\
			& \leq \displaystyle\sum\limits_{Y_i\in \overline{\mathcal{Z}}} |E(Y_i)| + \displaystyle\sum\limits_{Y_i\in \mathcal{Z}} |E(Y_i)|-r \\
			& \leq \displaystyle\sum\limits_{i=1}^p |E(Y_i)| - r \\
			& \leq r(\ell-1)
		\end{align*}
		Hence $Q'$ is a path between $u_1$ and $u_{\ell+1}$ of length at most $r(\ell-1)$. Hence $\dist{u_1}{u_{\ell+1}}$ is at most $r(\ell-1)$. However, this contradicts \Cref{obs:iso:lb}.
	\end{subproof}
	
	Now for each $Y_i\in \overline{\mathcal{Z}}$, define $Y'_i=Y_i$ and for each $Y_i\in \mathcal{Z}$, define $Y'_i=Y_i \setminus \{b_i\}$. 
	Observe that each path in $\mathcal{Y}'=\{Y'_i\colon Y_i\in \mathcal{Z}\} \cup \{Y'_i\colon Y_i\in \overline{\mathcal{Z}}\}$ 
	is an isometric path in $G$. Due to \Cref{clm:upper_bound}, $|\mathcal{Y}'|\leq 2r-1$. 
	Also, $V(P)\subseteq V(Q)\subseteq \displaystyle\bigcup\limits_{Y\in \mathcal{Y}'} V(Y)$. Therefore, the vertices of $P$ can be covered by $2r-1$ many isometric paths in $G$.
\end{proof}

\smallskip\noindent\textbf{Completion of the proof of \Cref{thm:line-power}\ref{it:pow}.} Let $G$ be a graph and $H$ be the $r^{th}$ power of $G$. Let $v$ be a vertex of $G$, and $k=\ipcor{v}{G}$. Let $P$ be an isometric path of $H$. 
Due to \Cref{lem:black}, $P$ contains at most one black edge $e$. Let $P_1$ and $P_2$ be the two connected components obtained by removing $e$ from $P$. Note that both $P_1$ and $P_2$ are red isometric paths. Due to \Cref{lem:isometric-cover-power}, for each $i\in \{1,2\}$, the vertices of $P_i$ can be covered by $2r-1$ isometric paths of $G$. 
Hence, the vertices of $P$ can be covered by $4r-2$ isometric paths of $G$. This implies that the vertices of $P$ can be covered by $k(4r-2)$ many $v$-rooted isometric paths of $G$. Thus, \Cref{lem:pow-2} implies that {$P$ can be covered by} $k(4r^2-2r)$ $v$-rooted isometric paths of $H$. The above arguments imply $\sipco{H}\leq (4r^2-2r)\sipco{G}$ and $\ipco{H}\leq (4r^2-2r)\ipco{G}$.

\subsection{Line graphs}\label{sec:line} 

In this section, we show that the (strong) isometric path complexity of the line graph of a graph $G$ is upper bounded by a constant factor of the (strong) isometric path complexity of $G$. First, we prove some lemmas.

\begin{lemma}\label{lem:line-original}
        Let $G$ be a graph and $P=e_1 e_2 \ldots e_{k}$ with $k\geq 2$ be an isometric path in $\linegraph{G}$. Then $e_2, e_3, \ldots,e_{k-1}$ induces an isometric path in $G$.
\end{lemma}

\begin{proof}
Let $Q$ be the path induced by the edges $e_2, e_3, \ldots,e_{k-1}$ in $G$ and let $u,v$ be the end-vertices of $Q$. If $Q$ is not an isometric path, then another $(u,v)$-path $Q'$ in $G$ of length at most $k-3$ exists. Let $e'_1,\ldots,e'_{t}$ with $t\leq k-3$ be the edges of $Q'$ and notice that $u$ (resp. $v$) is an end-vertex of $e'_1$ (resp. $e'_{t}$). Then $e_2 e'_1\ldots e'_{t} e_{k-1}$ is a shorter path in $\linegraph{G}$, a contradiction. 
\end{proof}

\begin{lemma}\label{lem:original-line}
Let $P$ be an isometric path in $G$. Then $\linegraph{P}$ is an isometric path in $\linegraph{G}$.
\end{lemma}

\begin{proof}
Let the edges of $P$ be $e_1, e_2,\ldots, e_k$ and the two end-vertices of $P$ are $u,v$. Then $\linegraph{P}$ has length $k-1$ in $\linegraph{G}$. If $\linegraph{P}$ is not an isometric path, then there exists a $(e_1,e_k)$-path in $\linegraph{G}$ that can be written as $e_1=e'_1 e'_2 \ldots e'_t=e_k$ such that $t\leq k-2$. Then by \Cref{lem:line-original}, we have that the edges $e'_2,\ldots,e'_{t-1}$ induce an isometric path $Q$ in $G$ of length $k-4$. Now, including $e_1$ and $e_k$ creates a $(u,v)$-path of length at most $k-2$, a contradiction.
\end{proof}

\begin{lemma}\label{lem:line-decompose}
    Let $G$ be a graph, $P$ be an $(u,v)$-isometric path in $G$ and $e_0$ be any edge incident to $u$ in $G$. Then, the vertices of $\linegraph{P}$ can be covered by two $e_0$-rooted isometric paths in $\linegraph{G}$.   
\end{lemma}

\begin{proof}
Let the edges of $P$ be ordered as $e_1,e_2,\ldots,e_k$ as they are traversed starting from $u$. We also assume that the end-vertices of $e_i$ are $u_{i}$ and $u_{i+1}$ for $i\in [k]$ where $u$ is closer to $u_{i}$ than $u_{i+1}$. Since $\linegraph{P}$ is an isometric path in $\linegraph{G}$ (due to \Cref{lem:original-line}), if $e_1=e_0$, then the statement is true. Now assume $e_1\neq e_0$. Observe that if $e_0$ shares an end-vertex with any edge $e_j$ with $j\geq 2$, then $P$ is not an isometric path in $G$. Now, assume that $i$ is the minimum index such that the path induced by $\{e_0, e_1, \ldots e_i\}$ is not an isometric path in $\linegraph{G}$. From previous arguments, it follows that $i\geq 3$. 

Let $Q=e_0 e'_1 \ldots e'_{t} e_i$ be an $(e_0,e_i)$-isometric path in $\linegraph{G}$. Observe that $t\leq i-2$ and by \Cref{lem:line-original}, the path induced in $G$ by the end-vertices of the edges $e'_1,\ldots,e'_t$ is an isometric path $P'$ in $G$. If the common vertex of $e'_{t}$ and $e_i$ is $u_{i+1}$, then by appending $u$ with $P'$, it is possible to construct an $(u,u_{i+1})$-isometric path of length at most $i-1$ in $G$, which contradicts that $P$ is an isometric path. Hence, the common vertex of $e'_{t}$ and $e_i$ is $u_{i}$. Let $Q'$ denote the path $e_0 e'_1\ldots e'_t e_i e_{i+1}\ldots e_k$ in $\linegraph{G}$ and observe that the length of $Q'$ is at most $k-1$. 

We argue that $Q'$ is an isometric path in $\linegraph{G}$. If not, then there exists another isometric path $Q''=e_0 e''_1 \ldots e''_{x} e_k$ in $\linegraph{G}$ such that $x\leq k-3$. Due to \Cref{lem:line-original}, we know that the end-vertices of the edges $e''_1,\ldots,e''_{x}$ induce an isometric path $P''$ in $G$. By including $u$ and $u_{k+1}$ to $P''$ it is possible to construct a $(u,u_{k+1})$-path in $G$ whose length is at most $k-1$. This is a contradiction. 

Finally, define $P'=e_0e_1\ldots, e_{i-1}$. The definition of $i$ implies that $P'$ is an isometric path in $\linegraph{G}$. Observe that $P'$ and $Q'$ are two $e_0$-rooted isometric paths in $\linegraph{G}$ that cover $\linegraph{P}$.
\end{proof}

\begin{figure}
    \centering
    \begin{tabular}{cc}
         \begin{tikzpicture}

          \foreach \x/\y [count = \n] in {
          0/0.5,0/0, 0/-0.5, 0/-1, 0/-1.5, 0/-2, 
          -0.5/-0.5, -0.5/-1, -0.5/-1.5, -0.5/-2,
          -1/-0.5, -1/-1,-1/-1.5, -1/-2,
          0.5/-0.5, 0.5/-1, 0.5/-1.5, 0.5/-2,
          1/-0.5,1/-1, 1/-1.5, 1/-2, 
          -0.75/-2.25, -0.25/-2.25, 0.25/-2.25, 0.75/-2.25, 1.25/-2.25
          }
    {
          \filldraw (\x, \y) circle (1pt);
    	 
    }

    \draw[thick] (0,0.5) -- (0,0) -- (0.5,-0.5) -- (0.5,-2);
    \draw[thick] (0,0) -- (-0.5,-0.5) -- (-0.5,-2);
    \draw[thick] (0,0) -- (-1,-0.5) -- (-1,-2);
    \draw[thick] (0,0) -- (1,-0.5) -- (1,-2);
    \draw[thick] (0,0) -- (0,-0.5) -- (0,-2);

    \draw  (-1,-2) -- (-0.75,-2.25) --  (-0.5,-2) -- (-0.25,-2.25) -- (0,-2) -- (0.25,-2.25) -- (0.5,-2) -- (0.75,-2.25)-- (1,-2) -- (1.25,-2.25) ;

    \node[left] at (0,0.1) {\scriptsize $r$}; 
    \node[left] at (0,0.5) {\scriptsize $r_0$}; 
    \foreach \x/\y [count = \n] in {
          -1/-2,-0.45/-1.9, 0.05/-1.9, 0.55/-1.9, 1.05/-1.9
          }
    {
          \node[left] at (\x,\y) {\scriptsize $b_{\n}$};
    	 
    }

    \foreach \x/\y [count = \n] in {
          -1/-1.5,-0.45/-1.5, 0.05/-1.5, 0.55/-1.5, 1.05/-1.5
          }
    {
          \node[left] at (\x,\y) {\scriptsize $b'_{\n}$};
    	 
    }

 \foreach \x/\y [count = \n] in {
          -1/-0.52,-0.45/-0.52, 0.05/-0.52, 0.55/-0.52, 1.05/-0.52
          }
    {
          \node[left] at (\x,\y) {\scriptsize $a_{\n}$};
    	 
    }

    \foreach \x/\y [count = \n] in {
          -0.75/-2.25,-0.25/-2.25, 0.25/-2.25, 0.75/-2.25, 1.25/-2.25
          }
    {
          \node[below] at (\x,\y) {\scriptsize $c_{\n}$};
    	 
    }

       \end{tikzpicture} & \begin{tikzpicture}
           \foreach \x/\y [count = \n] in {0/0, 0/0.5, -0.5/0, -1/0, 0.5/0, 1/0, -1/-0.5 , -0.5/-0.5, 0/-0.5, 0.5/-0.5, 1/-0.5, -1/-1 , -0.5/-1, 0/-1, 0.5/-1, 1/-1, -1.5/-1.5 , -1/-1.5, 0/-1.5, 1/-1.5, 1.5/-1.5, -1.75/-2, -1.25/-2, -0.75/-2, -0.25/-2, 1.75/-2, 1.25/-2, 0.75/-2, 0.25/-2, 2.25/-2 }
    {
          \filldraw (\x, \y) circle (1pt);
    	
    }

     \foreach \x/\y [count = \n] in {-1.75/-2, -1.25/-2,  -0.25/-2, 0.75/-2, 1.75/-2}
    {
          \node[below] at (\x,\y) {$e_{\n}$};
    	
    }

    \node [below] at (-0.75,-2) { $e'_2$ };
    \node [below] at (0.25,-2) { $e'_3$ };
    \node [below] at (1.25,-2) { $e'_4$ };
    \node [below] at (2.25,-2) { $e'_5$ };
    
    \foreach \x/\y [count = \n] in {
          -1.5/-1.5,-0.45/-1.5, 0.05/-1.5, 1.05/-1.5, 2/-1.5
          }
    {
          \node[left] at (\x,\y) {\scriptsize $f_{\n}$};
    	 
    }

    \foreach \x/\y/\w/\z [count = \n] in {0/0.5/-1/-0, 0/0.5/-0.5/0, 0/0.5/0/0, 0/0.5/0.5/0,0/0.5/1/0, -1/0/-1/-1, -0.5/0/-0.5/-1,0/0/0/-1, 0.5/0/0.5/-1, 1/0/1/-1, -1/-1/-1.5/-1.5, -0.5/-1/-1/-1.5, 0/-1/0/-1.5, 0.5/-1/1/-1.5, 1/-1/1.5/-1.5, -1.5/-1.5/-1.75/-2 , -1.75/-2/2.25/-2, -1/-1.5/-1.25/-2, -1/-1.5/-0.75/-2, 0/-1.5/-0.25/-2 ,  0/-1.5/0.25/-2,  1/-1.5/0.75/-2, 1/-1.5/1.25/-2, 1.5/-1.5/1.75/-2, 1.5/-1.5/2.25/-2   }
    {
          \draw[thick] (\x, \y) -- (\w,\z);
    	
    }
        \filldraw[fill=lightgray,opacity=0.5] (0,0) ellipse (1.25cm and 0.25cm); \node[right] at (1.25,0) {$K_5$} ;

        \node[above] at (0,0.5) {$(r,r_0)$}; 
       \end{tikzpicture}  \\
       (a)  & (b)
    \end{tabular}
    \caption{(a) The graph $F_5$; (b) the graph $\linegraph{F_5}$. }
    \label{fig:enter-label}
\end{figure}
\begin{lemma}\label{lem:edge-cover}
	Let $G$ be a graph and $r$ be a vertex such that $\ipcor{r}{G} =k$. For any isometric path $P$, there exists a set $\mathcal{P}$ of $r$-rooted isometric paths with $|\mathcal{P}| \leq k$ such that $|E(P) \setminus \displaystyle\bigcup\limits_{Q\in \mathcal{P}} E(Q)|\leq k-1$. 
\end{lemma}

\begin{proof} 
	Since $\ipcor{r}{G} =k$, (by definition) the vertices of $P$ can be covered by at most $k$ many $r$-rooted isometric paths. Let $\mathcal{P}$ be a set of $r$-rooted isometric paths with $|\mathcal{P}|\leq k$ such that $V(P) \subsetneq \displaystyle\bigcup\limits_{Q\in \mathcal{P}} V(Q)$ and $|E(P) \setminus \displaystyle\bigcup\limits_{Q\in \mathcal{P}} E(Q)|$ is minimized. In other words, $\mathcal{P}$ is a set of $k$ many $r$-rooted isometric paths that cover all vertices of $P$ and cover as many edges of $P$ as possible. Observe that, for each $Q\in \mathcal{P}$, the vertices $V(Q)\cap V(P)$ induces a subpath of $P$. Since $|\mathcal{P}|\leq k$, the lemma follows. 
\end{proof}

\begin{lemma}\label{lem:line-upper}
    Let $G$ be a graph and $r\in V(G)$ be a vertex with $\ipcor{r}{G} = k$. Then, for any edge $e\in E(G)$ incident on $r$, $\ipcor{e}{\linegraph{G}} \leq 3k+1$.
\end{lemma}
\begin{proof}
    Let $Q=e_1e_2\ldots e_{x}$ be an isometric path in $\linegraph{G}$. Due to \Cref{lem:line-original}, $e_2,e_3,\ldots e_{x-1}$ induces an isometric path $P$ in $G$. Due to \Cref{lem:edge-cover}, there exists a set $\mathcal{P}$ of $r$ rooted isometric paths in $G$ such that $|\mathcal{P}|\leq k$ and $|E(P) \setminus \displaystyle\bigcup\limits_{P'\in \mathcal{P}} E(P')|\leq k-1$. Let $e'_1,e'_2,\ldots e'_t$ with $t\leq k-1$ are the edges of $P$ that do not belong to any path in $\mathcal{P}$. 
    Due to \Cref{lem:original-line,lem:line-decompose}, for each $P'\in \mathcal{P}$, $\linegraph{P'}$ is an isometric path in $\linegraph{G}$ and it can be decomposed into two $e$-rooted isometric paths, say $X(P')$ and $Y(P')$, in $\linegraph{G}$. Now consider the set $\mathcal{Q}=\displaystyle\bigcup\limits_{P'\in \mathcal{P}} X(P')\cup Y(P')$. Notice that, it is possible that none of $E=\{e_1, e'_1,\ldots, e'_{t},e_x\}$ may lie in any path of $\mathcal{Q}$. For each $f\in E$, let $Z(f)$ denote an $(e,f)$-isometric path in $\linegraph{G}$. Clearly the set $\mathcal{Q}'= \mathcal{Q} \displaystyle\bigcup\limits_{f\in E} Z(f)$ contains only $e$-rooted isometric paths, $|\mathcal{Q}'|\leq 3k+1$ and any vertex of $Q$ lies in some isometric path in $\mathcal{Q}'$.
\end{proof}

\Cref{lem:line-upper} implies the following.
\begin{lemma}\label{lem:line-upper-sipco}
	Let $G$ be a graph. Then $\ipco{\linegraph{G}} \leq 3\cdot\ipco{G}+1$ and $\sipco{\linegraph{G}} \leq 3\cdot \sipco{G}+1$.
\end{lemma}

\begin{lemma}\label{lem:line-lower}
    For an integer $t\geq 1$, there exists a graph $G_t$ such that $\ipco{G_t}=t$ and $\ipco{\linegraph{G_t}}=2t-1$.
\end{lemma}

\begin{proof}
For $t\geq 1$, the graph $F_t$ is defined in three steps as follows: \begin{enumerate*}[label=\textbf{(\Alph*)}]
    \item consider a path $P_t$ of length $2t-1$ and name the vertices $b_1,c_1,\ldots,b_t,c_t$ where $b_1$ and $c_t$ are the two end-vertices of $P_t$;

    \item consider a new vertex $r$ and for each $1\leq i\leq t$, introduce a new path $Q_i$ of length $t-1$ between $r$ and $b_i$;

    \item introduce a new vertex $r_0$ adjacent to $r$.
\end{enumerate*} There are no other edges in $F_t$ except the ones mentioned above. Sometimes we shall call the paths $Q_1,Q_2,\ldots, Q_t$ \emph{pillars}. Now construct $G_t$ by considering two disjoint copies (say $C^1$ and $C^2$) of $F_t$ and identify the vertices with label $r_0$. 

We claim that $\ipco{G_t}$ is $t$. Let $a\in V(G)$ be any vertex of $G_t$, and without loss of generality, assume that $a\in V(C^1)$. Consider the copy of the path $P_t$ in $C^2$ and call it $S$. Observe that $S$ is an isometric path in $G_t$ and for any $\{i,j\}\subseteq [t]$, $\dist{a}{b_i}=\dist{a}{b_j}$. Hence, for any $\{i,j\}\subseteq [t]$ there is no $(a,b_i)$-isometric path in $G_t$ that contains $b_j$. Therefore $\ipco{G_t}\geq t$. 

Conversely, for $i\in \{1,2\}$, let $v^i$ denote the copy of the vertex $v\in V(F_t)$ in $C^i$. (Recall that $C^i$ is isomorphic to $F_t$.) Let $T$ be any isometric path in $G_t$. If $T$ does not contain the vertex $r_0$, it must lie completely in $C^1$ or completely in $C^2$. Without loss of generality, assume $T$ lies in $C^1$. For $j\in [t]$, let $S_j$ denote a $(r_0, c^1_j)$-isometric path and let $\mathcal{S}=\{S_1,S_2,\ldots,S_t\}$. Then the vertices of $T$ can be covered by the paths in $\mathcal{S}$. Now consider the case when $T$ contains $r_0$ and observe that the vertices of $T$ can be covered by two $r_0$-rooted isometric paths in $G_t$. Hence $\ipco{G_t}\leq t$.

Now we claim that $\ipco{\linegraph{G_t}}$ is $2t-1$. Observe that $\linegraph{G_t}$ is constructed by taking $\linegraph{C^1}$ and $\linegraph{C^2}$ and adding an edge between the vertices corresponding to the edges $e_0=rr_0$. For $i\in\{1,2\}$, $\linegraph{C^i}$ can be described as follows. \begin{enumerate}[label=\textbf{(\alph*)}]
    \item Consider a clique $X$ of size $t+1$ whose vertices correspond to the edges $rr_0, \{ra_j\}_{j\in [t]}$. Consider the pillar $Q_j$ and $a'_j$ be the vertex adjacent to $a_j$ in $Q_j$.  Observe that $ra_j$ is adjacent to $a_ja'_j$ in $\linegraph{C^i}$. Let $H$ denote the union of $X$ and the edges between $ra_j$ and $a_ja'_j$ in $\linegraph{C^i}$.
    
    \item \label{it:16B} For each $j\in [t]$, let $b'_j$ denote the vertex adjacent to $b_j$ in $Q_j$ and $f_j$ denote the edge $b_jb'_j$. Let $e_1$ denote the edge $b_1c_1$, and for $j\in [2,t]$, let $e_j, e_{j'}$ denote the edges $b_jc_{j-1}$ and $b_jc_j$, respectively. Observe that $\linegraph{P_t}$ consists of the vertices $e_1,e_2,e'_2,\ldots,e_t,e'_t$. Moreover, $f_1$ is adjacent to $e_1$ in $\linegraph{G_t}$ and for each $j\in[2,t]$, $f_j$ is adjacent to both $e_j,e'_j$ in $\linegraph{G_t}$. 
    \item Finally, for each $j\in [t]$, consider the path $Q_j$. Then $\linegraph{C^i}$ is constructed by taking the union of $H, \left\{\linegraph{Q_j}\colon j\in [t]\right\}$ and $\linegraph{P_t}$. 
\end{enumerate}

We prove the following claims.

\begin{claim}\label{clm:1}
   For $i\in \{1,2\}$, the copy of $\linegraph{P_t}$ in $\linegraph{C^i}$ (say $S^i)$ is an isometric path in $\linegraph{G_t}$. Moreover, the vertices $S^i$ are equidistant from the vertex $e_0$ in $\linegraph{G_t}$. 
\end{claim}
\begin{subproof}
    Fix an $i\in\{1,2\}$. Since the copy of $P_t$ in $C^i$ is an isometric path in $G_t$, due to \Cref{lem:original-line}, 
    $S_i$ is an isometric path in $\linegraph{G_t}$. Consider a pillar $Q_j$, $j\in [t]$, in $C^i$. Observe that $f_j$ (which corresponds to the edge $b_jb'_j$ in $G_t$) is at distance $t$ from $e_0$ in $\linegraph{G_t}$. For $j=1$, $e_j$ is at distance $t+1$ from $e_0$ and for $j\in [2,t]$, both $e_j,e'_j$ are at distance $t+1$ from $e_0$.
\end{subproof}

Using the above claim, we complete the proof. Let $e\in V(\linegraph{G_t})$ and without loss of generality assume $e\in V(\linegraph{C_1})$. Let $S$ denote the copy of the path $\linegraph{P_t}$ in $\linegraph{C_2}$. Due to \Cref{clm:1}, $S$ is an isometric path in $\linegraph{G_t}$ and all vertices of $\linegraph{P_t}$ are equidistant from the vertex $e_0=rr_0$ in $\linegraph{G_t}$. Moreover, there is no isometric path in $\linegraph{G_t}$ that contains two vertices of $S$. Therefore, at least $2t-1$ many $e$-rooted isometric paths are required to cover the vertices of $S$. Hence $\ipco{\linegraph{G_t}} \geq 2t-1$. 

Conversely, let $T$ be an isometric path in $\linegraph{G_t}$. If $T$ does not contain the vertex $e_0$, then the vertices of $T$ lie completely in $\linegraph{C_1}$ or completely in $\linegraph{C_2}$. In either case, the vertices of $T$ can be covered by $2t-1$ many $e_0$-rooted isometric paths. If $T$ contains the vertex $e_0$, then $T$ can be covered by two $e_0$-rooted isometric paths in $\linegraph{G_t}$. Hence $\ipco{\linegraph{G_t}}\leq 2t-1$.
\end{proof}

The proof of \Cref{thm:line-power}\ref{it:line} follows from \Cref{lem:line-upper,lem:line-upper-sipco,lem:line-lower}.

\subsection{Clique-sum}\label{sec:clique}

In this section, we prove \Cref{thm:clique-sum}. We begin by proving the following lemma.

\begin{lemma}\label{lem:glue-cover}
    Let $H$ be the clique-sum of two graphs, $H_1$ and $H_2$, with $C$ being the common clique. Let $P$ be an isometric path of $H$ such that $V(P)\subseteq V(H_1)$, $P$ has an end-vertex in $C$, and $w$ be any vertex of $H_2$. Then, the vertices of $P$ can be covered by three $w$-rooted isometric paths in $H$.
\end{lemma}

\begin{proof}
 By definition $C\subseteq V(H_1)$ and $C\subseteq V(H_2)$.  Assume for contradiction that there are no three $w$-rooted isometric paths in $H$ that cover all vertices of $V(P)$. 
 Let $u$ and $v$ be the end-vertices of $P$ and $v\in C$. For some isometric path $T$ and two vertices $\{x,y\}\subseteq V(T)$, $\subpath{T}{x}{y}$ shall denote the subpath of $T$ between $x$ and $y$. Let $u_1=u$, and for $i\in \{1,2,3\}$, let $Q_i$ be a vertex-minimal $w$-rooted isometric path that contains $u_i$ and covers the maximum number of vertices of $P$, and define $u_{i+1}$ to be the vertex in $V(P) \setminus \displaystyle\bigcup\limits_{j\leq i} V(Q_j)$ that is at minimum distance from $u$. For $i\in \{1,2,3,4\}$, let $v_i\in V(Q_i)\cap C$ that is at minimum distance from $w$. See \Cref{fig:lem:3-path} for illustration of the notations.

 \begin{figure}
     \centering
     \begin{tikzpicture}
     \definecolor{lightgray}{rgb}{0.83, 0.83, 0.83}
         \draw (0,0) rectangle (-6,2.75); \draw (-0.75,0.25) rectangle (2.25,3); 
         \draw[fill=lightgray,opacity=0.25] (-0.75,0.25) rectangle (0,2.75);
         \node[left] at (-6,1) {$H_1$}; \node[right] at (2.25,1) {$H_2$}; \node[above] at (2,2.5) {$w$};
         \draw[thick, dashed] (-0.5,0.5) -- (-5.5,0.5);

         \foreach \x/\y [count = \n] in {-0.5/0.5,-5.5/0.5, 2/2.5, -0.5/2.5, -0.5/2, -0.5/1.5, -4/0.5, -2.5/0.5, -4.5/0.5, -3/0.5, -1.5/0.5, -1/0.5}
			{
				 \filldraw (\x, \y) circle (1.5pt);
				 
			}
         \foreach \x/\y [count = \n] in {-5.5/0.5, -4/0.5, -2.5/0.5, -1/0.5}
			{
				\node[below] at (\x,\y) {$u_{\n}$};
				 
			}

            \foreach \x/\y [count = \n] in {-3/1.75, -2.25/1.5, -1.4/1.25}
			{
				\node[below] at (\x,\y) {$Q_{\n}$};
				 
			}

        \foreach \x/\y [count = \n] in {-0.5/2.5, -0.5/2, -0.5/1.5}
			{
				\node[right,below] at (\x+0.1,\y) {\scriptsize $v_{\n}$};
				 
			}

            \foreach \x/\y [count = \n] in {-4.5/0.5, -3/0.5, -1.5/0.5}
			{
				\node[below] at (\x,\y) {\scriptsize $u'_{\n}$};
				 
			}
            
        \node[right] at (-0.5,0.5) { $v$};

         \foreach \x/\y/\w/\z/\a/\b [count = \n] in {-0.5/2.5/-4.5/0.5/-5.5/0.5, -0.5/2/-3/0.5/-4/0.5, -0.5/1.5/-1.5/0.5/-2.5/0.5}
			{
				\draw[thick, densely dotted] (2,2.5) -- (\x,\y) -- (\w,\z) -- (\a,\b);
				    
			}

          \foreach \x/\y/\w/\z [count = \n] in {-4.5/0.5, -3/0.5, -1.5/0.5}
			{
				\draw[thick] (\x,\y) -- (\x+0.5,\y);
				 
			}
        
     \end{tikzpicture}
     \caption{Notations used in the proof of \Cref{lem:glue-cover}. }
     \label{fig:lem:3-path}
 \end{figure}

\begin{claim}\label{clm:u1}
    The vertex $u$ is an end-vertex of $Q_1$. 
\end{claim}
\begin{subproof}
Observe that $V(Q_1)\cap V(P)$ must be a subpath of $P$, and the minimality of $Q_1$ implies that the end-vertex of $Q_1$ distinct from $w$ must be in $V(Q_1)\cap V(P)$. Let $x$ be the end-vertex of $Q_1$ which is distinct from $w$, and suppose that $x\neq u$. Let $\dist{v}{x}=a$ and as $x\neq u$, $\dist{v}{u}\geq a+1$. Then, $\dist{v_1}{x}\leq a+1$. Since $u$ lies in an isometric path between $v_1$ and $x$, it follows that $\dist{v_1}{u}\leq a$. This implies that $\dist{v}{u}=a+1$ and $\dist{v_1}{x}=a+1$. Since $\dist{v}{x}=a$ (by assumption), the path induced by $\{v_1\}\cup V(\subpath{P}{v}{x})$ also has length $a+1$ and is therefore a $(v_1,x)$-isometric path. Now, consider the path $T$ induced by $V(\subpath{Q_1}{w}{v_1})\cup V(\subpath{P}{v}{x})$. If $T$ is a $(w,x)$-isometric path, then, together with $\subpath{Q_1}{w}{u}$, it covers all vertices of $P$, which contradicts our assumption that $P$ cannot be covered by three $w$-rooted paths. Otherwise, it must be the case that $\subpath{T}{w}{v}$ is not isometric and, therefore, $\dist{w}{v}\leq \dist{w}{v_1}$. But in this case, the path obtained by appending a $(w,v)$-isometric path to $\subpath{P}{v}{u_1}$ provides the necessary contradiction.
\end{subproof}

Similar arguments, as used in the proof of \Cref{clm:u1}, imply the following.

 \begin{claim}\label{clm:special-1}
   The vertices $u_2$ and $u_3$ are end-vertices of $Q_2$ and $Q_3$, respectively.
 \end{claim}
 
 For $i\in \{1,2,3\}$, let $u'_i$ be the vertex in $V(P)\cap V(Q_{i})$ which is adjacent to $u_{i+1}$. Let $\dist{u'_1}{v} = a$, $\dist{w}{v_1} = b$, and $\dist{u'_2}{u_2}=\ell$. \Cref{clm:special-1} implies that the vertices $u_1=u,u'_1,u_2,u'_2, u_3,u'_3,u_4$ appear in that order when $P$ is traversed from $u$ to $v$.

 \begin{claim}\label{clm:a}
     $\dist{v_1}{u'_1} = a-1$. 
 \end{claim}
 \begin{subproof}
     Since $\dist{v}{u'_1}=a$ and $v$ is adjacent to $v_1$, we have $a-1\leq \dist{u'_1}{v_1}\leq a+1$. 
     
     Suppose $\dist{v_1}{u'_1} = a+1$. If $\dist{w}{v}\leq \dist{w}{v_1}$, then the path obtained by appending a $(w,v)$-isometric path with $\subpath{P}{v}{u_1}$ is isometric and along with $Q_1$ covers $P$, which is a contradiction. Hence, $\dist{w}{v}>\dist{w}{v_1}$. In this case, the vertices $V(\subpath{Q_1}{w}{v_1}) \cup V(P)$ induces a $(w,u_1)$-isometric path that covers strictly more vertices of $P$ than $Q_1$, which is a contradiction. 
     
     Now assume that, $\dist{v_1}{u'_1}=a$. Observe that the path given by $V(\subpath{P}{u_2}{v}) \cup V(\subpath{Q_1}{v_1}{w})$ cannot be isometric (otherwise two $w$-rooted isometric paths cover all vertices of $P$). Hence, $\dist{u_2}{w}\leq a+b-1$ and let $Z$ be any $(w,u_2)$-isometric path. Observe that $Z\cup \set{u'_1}$ is a $(w,u'_1)$-isometric path and therefore, $V(Z)\cup V(\subpath{P}{u_2}{u})$ is an $(w,u)$-isometric path that covers strictly more vertices of $P$ than $Q_1$, a contradiction.
 \end{subproof}
 
 \begin{claim}\label{clm:a+b-1}
     $\dist{u_2}{w} = a+b-1$. 
 \end{claim}
 \begin{subproof}
    \Cref{clm:a} implies that $a+b-2\leq \dist{w}{u_2}\leq a+b$. If $\dist{w}{u_2} = a+b-2$, then any $(w,u_2)$-isometric path concatenated with $\subpath{P}{u_2}{u}$ yields a $(w,u)$-isometric path that covers strictly more vertices of $P$ than $Q_1$, which is a contradiction. If $\dist{u_2}{w}=a+b$, then $V(\subpath{P}{u_2}{v}) \cup V(\subpath{Q_1}{v_1}{w})$ yields an isometric path that covers strictly more vertices of $\subpath{P}{u_2}{v}$ than $Q_1$, a contradiction.
 \end{subproof}

Recall that $\dist{u'_2}{u_2}=\ell$ and observe that $\dist{v}{u_3} = a-\ell-2$. 

  \begin{claim}\label{clm:a+b-2-ell}
     $\dist{u_3}{w} = a+b-\ell-1$. 
 \end{claim}
 \begin{subproof}
    \sloppy Observe that $\dist{u_3}{w}\leq a+b-\ell-1$ follows from the sentence before the claim. 
    If $\dist{u_3}{w} = a+b-\ell$, then $V(\subpath{Q_1}{v_1}{w}) \cup V(\subpath{P}{v}{u_3})$ yields an isometric path that covers strictly more vertices of $\subpath{P}{u_3}{v}$ than $Q_3$, a contradiction. If $\dist{u_3}{w}=a+b-\ell-2$, then \Cref{clm:a+b-1} implies that $V(\subpath{Q_3}{w}{u_3})\cup V(\subpath{P}{u'_2}{u_2})$ induces an $(w,u_2)$-isometric path that covers strictly more vertices than $Q_2$, a contradiction.
 \end{subproof}

 Now consider the path $T$ induced by $\subpath{Q_1}{w}{v_1} \cup \subpath{P}{v}{u_3}$. The length of $T$ is at most $b+a-\ell-1$ and therefore is an $(w,u_3)$-isometric path that covers strictly more vertices than $Q_3$, a contradiction. 
\end{proof}

Let $G_1,G_2,\ldots,G_t$ be graphs, each with strong isometric path complexity at most $k$ and $G=\displaystyle\bigoplus\limits_{i=1}^t G_i$. Let $P$ be any isometric path of $G$ and $v$ be any vertex of $G$. We shall show that the vertices of $P$ can be covered by at most $3k+18$ many $v$-rooted isometric paths. First, we make the following observations.

\begin{observation}\label{obs:gluing}
    For a subset $X\subseteq [t]$, connected components of $\displaystyle\bigoplus\limits_{i\in X} G_i$ are isometric subgraphs of $G$.
\end{observation}

\begin{observation}\label{obs:contig-path}
    For any integer $i\in [t]$, $V(P)\cap V(G_i)$ is either empty or induces a subpath of $P$.
\end{observation}

\begin{lemma}\label{lem:clique-intermediate}
    Let $w$ be a vertex such that there exists an integer $i\in [t]$ such that $V(P)\cap V(G_i)$ is non-empty and $w\in V(G_i)$. Then, the vertices of $P$ can be covered by $k+6$ many $w$-rooted isometric paths in $G$.
\end{lemma}
\begin{proof}
     \Cref{obs:contig-path} implies that $V(P)\cap V(G_i)$ induces a subpath, say $P_i$, of $P$. Let $Q$ be a connected component of $V(P)\setminus V(P_i)$. Observe that there exists a clique $C$ in $G$ that separates $Q$ from $V(G_i)\setminus V(C)$. Let $H_1$ and $H_2$ be the two connected components in $V(G)\setminus C$ that contain  $Q$ and $V(G_i)\setminus V(C)$, respectively. Let $H$ be the clique-sum of the graphs induced by $V(H_1)\cup C$ and $V(H_2)\cup C$. Hence, $C$ is the common clique.   By \Cref{obs:gluing}, $H$ is an isometric subgraph of $G$. 
     
     Observe that one end-vertex of $Q$ is adjacent to exactly one vertex of $C$, say $u$, of $P_i$. Let $Q^+$ denote the path induced by $V(Q)\cup \{u\}$. Observe that $Q^+$ is an isometric path of $H$, $V(Q^+)\subseteq V(H_1)$, and $w$ is a vertex of $H_2$. Now, \Cref{lem:glue-cover} implies that the vertices of $Q^+$ can be covered by three $w$-rooted isometric paths of $H$. 

     To complete the proof of the lemma, observe that there can be at most two connected components of $V(P)\setminus V(G_i)$. The above arguments imply that all vertices of $V(P)\setminus V(P_i)$ can be covered by six $w$-rooted isometric paths in $G$. Since $\sipco{G_i}\leq k$, the vertices of $P_i$ can be covered by at most $k$ many $w$-rooted isometric paths of $G$. Hence, $P$ can be covered by $k+6$ many $w$-rooted isometric paths of $G$. 
\end{proof} 

\noindent\textbf{Completion of the proof of \Cref{thm:clique-sum}.} Observe that if there exists an integer $i\in [t]$ such that $V(P)\cap V(G_i)$ is non-empty and $v\in V(G_i)$, then the theorem follows from \Cref{lem:clique-intermediate}. Otherwise, there exists an integer $i\in [t]$ such that $V(G_i)\cap V(P)$ is non-empty and there exists a clique $C$ such that $v$ and $P$ lies in different connected components in the graph induced by $V(G)\setminus C$. Let $v'$ be any vertex of $C$. \Cref{lem:clique-intermediate} implies that there is a set $\mathcal{P}$ of $k+6$ many $v'$-rooted isometric paths in $G$ such that all vertices of $P$ are covered by the paths in $\mathcal{P}$. Now, by \Cref{lem:glue-cover}, each path in $\mathcal{P}$ can be covered by three $v$-rooted isometric paths in $G$. Therefore, the vertices of $P$ can be covered by $3k+18$ many $v$-rooted isometric paths in $G$. 

\begin{remark}
  A possible approach to reduce the constants in \Cref{thm:clique-sum} would be to consider an integer $i$ such that $V(G_i)\cap V(P)$ is nonempty, a clique $C$ in $G_i$ such that $P$ and $v$ lie in different components of $V(G)\setminus C$, and a vertex $w\in C$ which is closest to $v$ in $G$. Then, one might be tempted to consider $k+6$ many $w$-rooted paths that cover $P$ and then append them with $(w,v)$-isometric paths in $G$. However, the resulting paths may not be isometric in $G$, and therefore the above natural approach does not seem to work directly.
\end{remark}
\vspace{-0.4cm}

\section{Conclusion}\label{sec:conclude}

We proved that $U_t$-asymptotic minor-free graphs have bounded strong isometric path complexity. 
On the other hand, it follows from \cite{chakraborty2026isometric} that for $k\geq 3$, $(2\times k)$-grid-asymptotic minor-free graphs have unbounded (strong) isometric path complexity. However, the construction of \cite{chakraborty2026isometric} contains 
vertices of large degree. On the other hand, the graphs constructed to prove \Cref{thm:even} have bounded maximum degree but do not yield $(2\times k)$-grid-asymptotic minor-free graphs for any constant $k$. We conjecture the following.
	
	\begin{conjecture}
	There is a function $f\colon \mathbb{N}\times \mathbb{N}\times \mathbb{N} \rightarrow \mathbb{N}$ such that the graphs with maximum degree $d$ and excluding $(2\times t)$-grids as a $K$-fat minor have (strong) isometric path complexity at most $f(d,t,K)$.
\end{conjecture}

We proved that monoholed graphs are $U_4$-asymptotic minor-free, which is of independent interest. Thus, these graphs have bounded strong isometric path complexity. For an integer $\ell\geq 1$, let us call a graph \emph{$\ell$-polyholed} if the total number of distinct lengths of the induced cycles in $G$ is at most $\ell$. We ask the following question.

\begin{Question}
	Is there a function $f\colon \mathbb{N}\rightarrow \mathbb{N}$ such that $\ell$-polyholed graphs with $\ell\geq 1$ have (strong) isometric path complexity at most $f(\ell)$?
\end{Question}

It is easy to construct a graph class whose strong isometric path complexity is small but contains all $U_t$-asymptotic minors. For $t\geq 1$, consider $K_{1,t}$ and replace its edges with a path of length $t$. Let $u_1,u_2,\ldots,u_t$ denote the leaves. Let $H_t$ denote the graph obtained by introducing a path of length $t$ between $u_i$ and $u_{i+1}$ for each $i\in [t-1]$. Observe that $\sipco{H_t}\leq 2$, but $\mathcal{H}=\{H_t\}_{t\in \mathbb{N}}$ contains $U_t$-asymptotic minors for all $t$. Therefore, {$U_t$-asymptotic minors do not provide} a structural characterization of graphs with bounded strong isometric path complexity. Since the property of bounded (strong) isometric path complexity is not preserved under quasi-isometry, we do not expect a coarse characterisation of graphs with bounded strong isometric path complexity. Nevertheless, it would be interesting to obtain a metric characterisation of graphs with bounded strong isometric path complexity. It would also be interesting to characterize graphs with small values of strong isometric path complexity.

We proved that important graph operations like fixed powers, line graphs, and clique-sum preserve strong isometric path complexity. 
However, it is not clear whether our upper bounds are tight. It would also be interesting to characterize which other graph operations preserve strong isometric path complexity.

Graphs with bounded strong isometric path complexity generalize both hyperbolic graphs and $U_t$-asymptotic minor-free graphs and their subclasses. They can be seen as an extension of hyperbolic graphs, where arbitrarily large isometric cycles are allowed to exist (although in a restricted fashion).
It is more generally of interest to see which interesting properties can be lifted from them upto graphs of bounded strong isometric path complexity. This could shed new light on these classes.
It would also be of interest to discover more algorithmic applications of (strong) isometric path complexity.

\bibliographystyle{plain}
\bibliography{references}
 
\end{document}